\documentclass[11pt]{article}
\usepackage{amsfonts}
\usepackage{color}
\usepackage{amsmath,amssymb,comment}
\newtheorem{theo}{Theorem}[section]
\newtheorem{lem}[theo]{Lemma}
\newtheorem{cor}[theo]{Corollary}

\newtheorem{defi}[theo]{Definition}

\newcommand{\mysection}[1]{\section{#1} \setcounter{equation}{0}}
\newcommand{\proof}{{\sc Proof.} \quad}

\newcommand{\be}{\begin{equation} \label}
\newcommand{\ee}{\end{equation}}
\newcommand{\bea}{\begin{eqnarray}\label}
\newcommand{\eea}{\end{eqnarray}}
\newcommand{\bas}{\begin{eqnarray*}}
\newcommand{\eas}{\end{eqnarray*}}
\newcommand{\bit}{\begin{itemize}}
\newcommand{\eit}{\end{itemize}}
\newcommand{\qed}{\hfill$\Box$ \vskip.2cm}
\newcommand{\nn}{\nonumber}

\newcommand{\eps}{\varepsilon}

\newcommand{\supp}{{\rm supp} \, }

\newcommand{\io}{\int_\Omega}

\newcommand{\Del}{\Delta}

\newcommand{\Om}{\Omega}

\newcommand{\ov}{\overline}

\newcommand{\lbal}{\left\{ \begin{array}{l}}
\newcommand{\lball}{\left\{ \begin{array}{ll}}
\newcommand{\ear}{\end{array} \right.}

\newcommand{\abs}{\\[5pt]}

\newcommand{\adb}{\allowdisplaybreaks}
\newcommand{\tm}{T_{max}}
\newcommand{\tme}{T_{max,\eps}}

\newcommand{\tu}{\tilde{u}}
\newcommand{\tT}{\tilde{\Theta}}
\newcommand{\tg}{\tilde{g}}
\begin{document}
\adb
%\enlargethispage{10mm}
%
%
\title{Well-posedness of a stress-augmented Moore-Gibson-Thompson-type model for thermoviscoelastic evolution in a standard linear solid}
\author{
Torben Fricke\footnote{tjfricke@math.uni-paderborn.de}\\
{\small Universit\"at Paderborn, Institut f\"ur Mathematik}\\
{\small 33098 Paderborn, Germany}}
\date{}
\maketitle
\begin{abstract}
\noindent 
This manuscript studies a coupled system of nonlinear parabolic and hyperbolic equations that models thermoviscoelastic materials with explicit internal stress relaxation,
\bas
	\lbal
	\tau u_{ttt} + u_{tt} = (D(\Theta)u_{xt})_x + (C(\Theta)u_x)_x,\\[1mm]
	\Theta_t = \Theta_{xx} + |E(\Theta) u_{xt} + F(\Theta) g_t|^2,\\[1mm]
	\tau g_t + g = D(\Theta) u_{xt} + C(\Theta) u_x.
	\ear
\eas
The system arises from a stress–strain constitutive relation of Zener type, where the internal variable $g$ encodes mechanical memory effects. A key feature is the explicit coupling of stress relaxation to the heat equation through the dissipation term, which differs from standard approaches that eliminate $g$ via algebraic substitution.\abs
Under the assumptions that $\tau>0$ and that $D,C\in C^2([0,\infty))$ are positive and $E,F\in C^1([0,\infty))$, and for suitably regular initial data, we establish local existence and uniqueness of strong solutions. Under additionally strengthened regularity requirements on the initial data and on $D,C$, this solution is moreover shown to be classical. These results extend a corresponding local strong and classical solvability theory developed for the underlying system without the stress variable $g$ to the present, stress-augmented setting.\abs
\end{abstract}
%
%
%
%
%35D30 Weak solutions to PDEs
%35L05 Wave equation
%74H20 Existence of solutions of dynamical problems in solid mechanics
%74F05 Thermal effects in solid mechanics
%
%
%
\newpage
\section{Introduction}\label{intro}
Beyond the classical theory of thermoviscoelasticity in Kelvin-Voigt materials, beginning in the 1980s \cite{Dafermos82} and continuing through more recent contributions \cite{BadalFriedrichKrucik23,BlanchardGuibe00,MielkeRoubicek20,RackeZheng97,Roubicek09}, the literature has since turned to models involving third-order time derivatives, as characteristic of Moore-Gibson-Thompson (MGT) type relaxation. Already at the linear level, such models exhibit a dichotomy between exponential stability and infinite-time grow-up, governed by the relative size of the parameters involved \cite{ChenIkehata21,DellOroPata17,KaltenbacherLasieckaMarchand11}; nonlinear MGT models have since been studied with regard to well-posedness \cite{KaltenbacherLasieckaPospieszalska12,KaltenbacherNikolic19,RackeSaidHouari21,SaidHouari22}, including a semigroup-based treatment under inclusion of memory effects \cite{AlvesCaixetaSilvaRodrigues18}, with related stability and decay results for the equation with memory obtained in \cite{DellOroLasieckaPata16,LasieckaWang15,LasieckaWang16}, and blow-up phenomena addressed in \cite{ChenPalmieri20,NikolicWinkler24} (see also \cite{ChenIkehata21} for the threshold between decay and blow-up in the dissipative case). Quasilinear couplings of a heat equation to Kelvin-Voigt type mechanical relaxation have moreover been examined with regard to local and global existence theory \cite{ClaesLankeitWinkler25,ClaesWinkler26b,Fricke25,Meyer26,Win26arxiv,Winkler25a}, as well as, more recently, with regard to strict positivity of the temperature \cite{FrickePositivity26,WinCalcVar26}; in both cases the heat equation is of the form obtained from \eqref{e10} below upon setting $F\equiv 0$. The quasilinear MGT system for thermoviscoelastic evolution in a standard linear solid recently studied in \cite{CW26} has subsequently been upgraded to classical solvability in \cite{FW26}, and, for the same system, to global solvability in \cite{BlackWinkler26}; this system forms the immediate starting point of the present manuscript, as detailed below.\abs
Structurally distinct from the Zener/standard-linear-solid framework considered here, Westervelt-Pennes-type systems couple acoustic wave propagation to temperature fields through a second- and, more recently, third-order MGT relative of \eqref{e3}-\eqref{e4} with a differently structured heat source; corresponding local and global well-posedness results have been obtained in \cite{BenabbasSaidHouari24,BenabbasSaidHouari25,BenabbasSaidHouari26,CareagaNikolicSaidHouari25,NikolicSaidHouari22a,NikolicSaidHouari22b}. As neither the mechanical equation nor the heat source in these systems follow a Zener-type argument, we do not draw on this literature further below.\abs
The present manuscript extends the work of \cite{CW26} and \cite{FW26}, by retaining the mechanical stress, which is eliminated algebraically from governing equations, resulting in a two-component system for displacement and temperature. This elimination rests on the assumption, made explicit in \cite{CW26}, that the elastic behavior is dominant as compared to mechanical losses, so that additional and temporally nonlocal contributions to the heat source can be neglected. We show below that retaining the stress as an independent variable $g$ allows this assumption to be dropped, reactivating a dissipative source term that couples the heat equation to the full mechanical relaxation dynamics.\abs 
To describe our specific object of study and its application background, let us consider a material in which the viscoelastic properties are described by a Zener-type model, also known as a standard linear solid (\cite{Gut14,Mai10,Zen48}). Writing $T$ and $S$ for the mechanical stress and strain, the Zener constitutive relation reads
\bea{e1}
T+\tau T_t=c(\tau dS_t+S),
\eea
with relaxation time $\tau>0$, stiffness ratio $d>1$ and the stiffness $c$. Combining this with the equation of motion and the definition of the mechanical strain in one-dimensional settings, we have
\bea{e2}
\rho u_{tt}=T_x\qquad\text{and}\qquad S=u_x,
\eea
where $\rho>0$ is the density and $u$ the displacement, and one obtains the Moore-Gibson-Thompson equation
\bea{e3}
\tau u_{ttt}+u_{tt}=\bigg(\frac{\tau cd}{\rho}u_{xt}\bigg)_x+\bigg(\frac{c}{\rho}u_x\bigg)_x.
\eea
Assuming that $c$ depends on the temperature field $\Theta$, after normalizing $\rho\equiv 1$ and setting $D:=\tau cd$, $C:=c$, we obtain the following parametrization for the differential equation for $u$:
\bea{e4}
\tau u_{ttt}+u_{tt}=(D(\Theta)u_{xt})_x+(C(\Theta)u_x)_x.
\eea
Identifying $g:=T$ and reading \eqref{e1} directly as an evolution equation for $g$ yields
\bea{e5}
\tau g_t+g=D(\Theta)u_{xt}+C(\Theta)u_x.
\eea
The mechanical stress $g$ can thus be retained as an independent unknown without altering the equation governing $u$; in contrast to the elimination in \cite{CW26}, $g$ remains available for coupling to the heat equation.\abs
Following \cite{CW26}, we start from the work $P=T\cdot S_t$, see \cite{BW12}. Inserting \eqref{e1} yields
\bea{e6}
P=cS\cdot S_t+c\tau d S_t^2-\tau S_t\cdot T_t.
\eea
The first term on the right is the stored potential mechanical energy in the material, which is easily seen to vanish on average for harmonic processes; the two remaining terms describe mechanical losses converted into heat. The physically motivated condition $d>1$ ensures that these loss terms combine, following the more refined calculation of \cite{Tau67}, into the manifestly nonnegative mechanical loss density
\bea{e7}
Q=\frac{\tau c}{d-1}\bigg(d S_t-\frac{1}{c}T_t\bigg)^2.
\eea  
In \cite{CW26}, the analogous step proceeds by additionally approximating $T_t\approx cS_t$, justified there under the assumption that elastic behavior dominates mechanical losses; this yields the simpler source $\Gamma(\Theta)u_{xt}^2$. We retain the exact expression above instead, made possible by keeping $g:=T$. Inserting $Q$ as a source into a one-dimensional diffusion equation and using $S_t=u_{xt}$, $T_t=g_t$, we obtain
\bea{e9}
\Theta_t=\Theta_{xx}+\frac{\tau c}{d-1}\bigg(d u_{xt}-\frac{1}{c}g_t\bigg)^2
\eea
with exclusively local sources, whence again assuming temperature-dependent elastic properties, we end up with
\bea{e10}
\Theta_t=\Theta_{xx}+|E(\Theta)u_{xt}+F(\Theta)g_t|^2.
\eea
Formally setting $F\equiv 0$ in \eqref{e10} reduces the source term to $E(\Theta)^2 u_{xt}^2$, matching the structure of the heat equation in \cite{CW26} with $\Gamma=E^2$. Supplementing \eqref{e4}, \eqref{e5} and \eqref{e10} with reasonable initial data and boundary conditions, we are thus led to considering the problem
\be{0}
	\lball
	\tau u_{ttt}+u_{tt}=(D(\Theta)u_{xt})_x+(C(\Theta)u_x)_x,& x\in\Om,\ t>0,\\[1mm]
	\Theta_t=\Theta_{xx}+|E(\Theta)u_{xt}+F(\Theta)g_t|^2,& x\in\Om,\ t>0,\\[1mm]
	\tau g_t+g=D(\Theta)u_{xt}+C(\Theta)u_x,& x\in\Om,\ t>0,\\[1mm]
	u_x=0,\quad \Theta_x=0,\quad g=0,& x\in\partial\Om,\ t>0,\\[1mm]
	(u,u_t,u_{tt},\Theta,g)(x,0)=(u_0,u_{0t},u_{0tt},\Theta_0,g_0)(x),\ & x\in\Om.
	\ear
\ee
Here, $\Om\subset\mathbb{R}$ is an open bounded interval, $\tau>0$ is a constant parameter, $D$, $C$, $E$ and $F$ are given functions on $[0,\infty)$, and $u_0$, $u_{0t}$, $u_{0tt}$, $\Theta_0$ and $g_0$ are prescribed initial distributions of the displacement variable $u=u(x,t)$, its derivatives $u_t$ and $u_{tt}$, the temperature $\Theta$, and the mechanical stress $g$.\abs
As already noted in \cite{CW26}, the functions $D$, $C$, $E$ and $F$ may reflect experimentally observed, possibly unbounded, temperature dependencies of viscoelastic parameters \cite{Fri23}, so that global solutions to \eqref{0} cannot be expected in general; cf. \cite{WinHotspot} for a corresponding blow-up result in a related Kelvin-Voigt simplification.\abs
In formulating this and throughout the sequel, given a bounded interval $\Om\subset\mathbb{R}$ and $p\in(1,\infty]$ we let $W_N^{2,p}(\Om):=\{\varphi\in W^{2,p}(\Om)\ |\ \frac{\partial\varphi}{\partial\nu}=0\text{ on }\partial\Om\}$. We further note that dividing the first equation in \eqref{0} by $\tau$ shows that the equations governing $u$, $u_t$, $u_{tt}$ coincide with those studied in \cite{CW26} upon the substitution
\bea{e11}
\alpha:=\frac{1}{\tau},\qquad \gamma:=\frac{D}{\tau}\qquad\text{and}\qquad \hat{\gamma}:=\frac{C}{\tau}.
\eea
Together with the reduction of the heat equation noted above, this substitution makes precise the sense in which \eqref{0} generalizes the system of \cite{CW26}: the equations for $u$, $\Theta$ and $g$, with $g$ eliminated, are recovered exactly, and $g$ constitutes the genuinely new ingredient of the present system. We refer to this substitution throughout the manuscript whenever results from \cite{CW26} or \cite{FW26} are invoked for the equations governing $u$, $u_t$, $u_{tt}$ and $\Theta$.\abs
\textbf{Main results.\quad}With the above substitution, the local existence theory from \cite{CW26} extends to the present setting once complemented by a treatment of the relaxation equation for $g$, yielding the following statement on local existence and uniqueness of strong solutions, along with a suitable extensibility criterion:
\begin{theo}\label{theo1}
Let $\Om\subset\mathbb{R}$ be an open bounded interval, $\tau>0$, and suppose that
\bea{1.10}
D,C\in C^2([0,\infty))\text{ are positive on }[0,\infty)\qquad\text{and}\qquad E,F\in C^1([0,\infty)).
\eea
Then whenever $u_0,u_{0t}\in W_N^{2,2}(\Om)$, $u_{0tt}\in W^{1,2}(\Om)$, $g_0\in W^{1,2}(\Om)$ and
$\Theta_0\in W_N^{2,\infty}(\Om)$ with $\Theta_0\ge 0$ in $\Om$,
there exist $T_{max}\in(0,\infty]$ as well as a unique triple $(u,\Theta,g)$ of functions
\be{1.13}
	\lball
	u\in C^0([0,\tm);W_N^{2,2}(\Om)),\\[1mm]
	\Theta\in C^0([0,\tm);C^1(\ov\Om))\cap C^{2,1}(\ov\Om\times(0,\tm))\cap W_{loc}^{1,\infty}(\ov\Om\times[0,\tm))\qquad\text{and}\\[1mm]
	g\in C^1([0,\tm);C^0(\ov\Om))
	\ear
\ee
which are such that
\bea{1.14}
u_t\in C^0([0,\tm);C^1(\ov\Om))\cap L_{loc}^\infty([0,\tm);W_N^{2,2}(\Om))
\eea
and
\bea{1.15}
u_{tt}\in C^0(\ov\Om\times[0,\tm))\cap L_{loc}^\infty((0,\tm);W^{1,2}(\Om)),
\eea
that $\Theta\ge 0$ in $\Om\times(0,\tm)$, that $(u,\Theta,g)$ forms a strong solution of \eqref{0} in $\Om\times(0,\tm)$ in the sense of Definition \ref{def2.1} below, and that if $\tm<\infty$, then
\bea{1.16}
\limsup_{t\nearrow\tm}\Big\{\|u_t(\cdot,t)\|_{W^{2,2}(\Om)}+\|u_{tt}(\cdot,t)\|_{W^{1,2}(\Om)}+\|\Theta(\cdot,t)\|_{W^{2,\infty}(\Om)}\Big\}=\infty.\hspace{0mm}
\eea
\end{theo}
Following the approach of \cite{FW26}, enhanced regularity of the initial data and of $D$ and $C$ then entails the existence of classical solutions, via an $\eps$-independent bound on a higher-order quasi-energy functional, the only new ingredient being a corresponding control of $\Theta_{xxx}$ in the presence of the coupling to $g$. This leads to the following counterpart of Theorem \ref{theo1}:
\begin{theo}\label{theo2}
In the setting of Theorem \ref{theo1}, suppose additionally that $D,C\in C^3([0,\infty))$ and that \bea{1.20}
u_0,u_{0t}\in W^{3,2}(\Om)\qquad\text{and}\qquad u_{0tt}\in W_N^{2,2}(\Om).
\eea Then there exist $\tm\in(0,\infty]$ as well as a unique triple $(u,\Theta,g)$ of functions
\be{1.21}
	\lball
	u\in C^0([0,\tm);C^2(\ov\Om))\cap L^\infty((0,\tm);W^{3,2}(\Om)),\\[1mm]
	\Theta\in C^0([0,\tm);C^1(\ov\Om))\cap C^{2,1}(\ov\Om\times(0,\tm))\cap W_{loc}^{1,\infty}(\ov\Om\times[0,\tm))\qquad\text{and}\\[1mm]
	g\in C^1([0,\tm);C^0(\ov\Om))
	\ear
\ee
such that
\bea{1.22}
u_t\in C^0([0,\tm);C^2(\ov\Om))\cap L_{loc}^\infty((0,\tm);W^{3,2}(\Om)),\\
u_{tt}\in C^0([0,\tm);C^1(\ov\Om))\cap L_{loc}^\infty((0,\tm);W^{2,2}(\Om)),\label{1.23}
\eea
as well as
\bea{1.24}
u_{ttt}\in C^0(\ov\Om\times[0,\tm))\cap L_{loc}^\infty((0,\tm);W^{1,2}(\Om))
\eea
and that $\Theta\ge 0$ in $\Om\times(0,\tm)$, that $(u,\Theta,g)$ solves \eqref{0} in the classical sense in $\Om\times(0,\tm)$, and that, if $\tm<\infty$, then
\bea{1.25}
\limsup_{t\nearrow\tm}\Big\{\|u_t(\cdot,t)\|_{W^{3,2}(\Om)}+\|u_{tt}(\cdot,t)\|_{W^{2,2}(\Om)}+\|\Theta(\cdot,t)\|_{W^{2,\infty}(\Om)}\Big\}=\infty.
\eea
\end{theo}
\textbf{Remark.\quad}We note that neither extensibility criterion, \eqref{1.16} in Theorem \ref{theo1} nor \eqref{1.25} in Theorem \ref{theo2}, involves $g$, even though the system now comprises the additional equation for $g$. This is a consequence of the structure of \eqref{0} itself: since $g$ solves the linear relaxation equation $\tau g_t+g=D(\Theta)u_{xt}+C(\Theta)u_x$, its size is controlled by that of $u_x$, $u_{xt}$ and $\Theta$ alone, via the variation-of-constants representation
$$g(x,t)=g_0(x)e^{-\frac{t}{\tau}}+\frac{1}{\tau}\int_0^t e^{-\frac{t-s}{\tau}}\big(D(\Theta(x,s))u_{xt}(x,s)+C(\Theta(x,s))u_x(x,s)\big)ds.$$ 
\textbf{Novel aspects.\quad}The genuinely new feature of \eqref{0}, compared to the systems studied in \cite{CW26} and \cite{FW26}, is the additional equation for $g$, which possesses no diffusive dynamics of its own. Consequently, as already noted in the Remark above, $g$ is algebraically bound to the history of $u_x$, $u_{xt}$ and $\Theta$, rather than governed by its own evolution mechanism, so that it can be reconstructed pointwise from $(u,\Theta)$ alone, making $g$ temporally nonlocal, yet requiring no diffusive regularization of its own.\abs
\textbf{Basic strategy.\quad}As in \cite{CW26} and \cite{FW26}, we approximate \eqref{0} by a family of parabolic regularizations, and derive a hierarchy of a priori bounds of increasing derivative order before passing to the limit. For $g$, since its relaxation equation possesses no diffusive term of its own, we regularize it with an artificial Dirichlet diffusion, and use an auxiliary function, standing in for $g_t$, to express the resulting source term in the equation for $\Theta$. Following the approach of \cite{CW26} and \cite{FW26}, the highest-order bounds are organized through a quasi-energy of the form
$$y(t)=\frac{1}{2}\io u_{xtt}^2+\frac{1}{2\tau}\io D(\Theta)u_{xxt}^2+\frac{1}{\tau}\io C(\Theta)u_{xx}u_{xxt}+\frac{B}{2}\io u_{xx}^2,$$
with suitably large $B>0$, which is shown to satisfy a differential inequality on a short time interval, uniformly with respect to the approximation; a comparison argument then yields the desired bounds in the limit.\abs Most of the individual estimates along the way are routine adaptations of the corresponding steps in \cite{CW26} and \cite{FW26} under the substitution described above, and we do not comment on these separately. Wherever $g$ enters the argument, however, genuinely additional work is required that has no counterpart there, simply because neither paper contains an equation for $g$ at all. This concern, in particular: the constant $M$ fixed in Lemma \ref{lem2.4}, which must now also control $g$ itself; the a priori bound on $g_x$ obtained in Lemma \ref{lem2.8}, on which the subsequent control of $g_{xx}$ and $u_{xxx}$ rest; the bounds on the auxiliary quantities entering the source terms are markedly larger and more numerous than in the corresponding heat equation, precisely because this source now depends on $g$ in addition to $u_{xt}$; and finally, the bound on $g_t$ established in Lemma \ref{lem2.12}. While each of these estimates follows the same testing philosophy as before, none of them can simply be read off from the literature, and thus they form the material this manuscript actually supplies.
\mysection{Strong solutions}\label{strong}
\subsection{Definition and uniqueness of strong solutions}
We begin by making precise the notation to be used throughout this paper. As already indicated in the introduction, the equation for $u$ coincides with the one studied in \cite{CW26}; where our arguments parallel the computations given there, we recall the relevant intermediate steps but do not reproduce their derivation.
\begin{defi}\label{def2.1}
Let $\tau>0$, suppose that $D,\ C,\ E,\ F\in C^0([0,\infty))$ and let $u_0\in W_N^{2,2}(\Om)$, $u_{0t}\in W_N^{2,2}(\Om)$, $u_{0tt}\in W^{1,2}(\Om)$, $\Theta_0\in C^0(\ov\Om)$ with $\Theta_0\ge0$ and $g_0\in C^0(\ov\Om)$ with $g_0=0$ on $\partial\Om$. Then given $T\in(0,\infty]$, we will call a triple $(u,\Theta,g)$ of functions
\be{2.1}
	\lball
	u\in C^0([0,T);C^1(\ov\Om))\cap L_{loc}^\infty([0,T);W_N^{2,2}(\Om)),\\[1mm]
	\Theta\in C^0(\ov\Om\times[0,T))\cap C^{2,1}(\ov\Om\times(0,T))\cap W_{loc}^{1,\infty}(\ov\Om\times[0,T))\qquad\text{and}\\[1mm]
	g\in C^1([0,T);C^0(\ov\Om))
	\ear
\ee
a strong solution of \eqref{0} in $\Om\times(0,T)$ if
\be{2.2}
	\lball
	u_t\in C^0([0,T);C^1(\ov\Om))\cap L_{loc}^\infty([0,T);W_N^{2,2}(\Om))\qquad\text{and}\\[1mm]
	u_{tt}\in C^0(\ov\Om\times[0,T))\cap L_{loc}^\infty([0,T);W^{1,2}(\Om)),
	\ear
\ee
if $\Theta\ge 0$ in $\Om\times(0,T)$, if
\bea{2.3}
u(\cdot,0)=u_0,\quad u_t(\cdot,0)=u_{0t},\quad \Theta(\cdot,0)=\Theta_0\quad\text{and}\quad g(\cdot,0)=g_0\quad\text{in }\Om,
\eea
if
\bea{2.4}
-\tau\int_0^T\io u_{tt}\varphi_t-\tau\io u_{0tt}\varphi(\cdot,0)+\int_0^T\io u_{tt}\varphi=-\int_0^T\io D(\Theta)u_{xt}\varphi_x-\int_0^T\io C(\Theta)u_x\varphi_x
\eea
for each $\varphi\in C_0^\infty(\ov\Om\times[0,T))$, and if in the classical pointwise sense we have
\bea{nn}
\Theta_t=\Theta_{xx}+|E(\Theta)u_{xt}+F(\Theta)g_t|^2\quad\text{as well as}\quad \tau g_t+g=D(\Theta)u_{xt}+C(\Theta)u_x\nn
\eea
in $\Om\times(0,T)$, and $\Theta_x=0$ on $\partial\Om\times(0,T)$.\vspace{2mm}
\end{defi}
\textbf{Remark.}\quad i)\quad Fix a family $(\zeta_\eta)_{\eta\in(0,T)}\subset C_0^\infty([0,T))$ such that $\zeta_\eta\equiv 1$ on $[0,\frac{\eta}{2}]$, $\zeta_\eta'\le 0$ and $\supp\zeta_\eta\subset[0,\eta]$ for each $\eta\in(0,T)$. Inserting $\varphi(x,t):=\zeta_\eta(t)$ for $x\in\ov\Om$, $t\in[0,T)$ and $\eta\in(0,T)$, into \eqref{2.4} shows that since $\sup_{t\in(0,\eta)}\|u_{tt}(\cdot,t)-u_{tt}(\cdot,0)\|_{L^\infty(\Om)}\to 0$ as $\eta\searrow 0$ by \eqref{2.2}, in addition to \eqref{2.3} we also have $u_{tt}(\cdot,0)=u_{0tt}$.\abs
ii)\quad If $D,\ C\in C^1([0,\infty))$, then \eqref{2.4} forces every strong solution to satisfy, in the distributional sense, $$\tau u_{ttt}+u_{tt}=D(\Theta)u_{xxt}+D'(\Theta)\Theta_x u_{xt}+C(\Theta)u_{xx}+C'(\Theta)\Theta_x u_x\qquad\text{in }\Om\times(0,T);$$
since $\Theta_x$ is bounded by \eqref{2.1}, this in particular yields $u_{ttt}\in L_{loc}^\infty([0,T);L^2(\Om))$.\abs
iii)\quad Every strong solution automatically satisfies $g(\cdot,t)=0$ on $\partial\Om$ for all $t\in[0,T)$: for $x_0\in\partial\Om$, the boundary condition on $u$ gives $u_x(x_0,\cdot)\equiv 0$ on $(0,T)$ and hence also $u_{xt}(x_0,\cdot)\equiv 0$ on $(0,T)$, so that $h:=g(x_0,\cdot)$ solves the initial value problem $\tau h_t+h=0$ on $(0,T)$, $h(0)=g_0(x_0)=0$, which has only the trivial solution, so that $h\equiv 0$ on $[0,T)$.
\begin{lem}\label{lem2.2}
Let $\tau>0$ and suppose that $D,C,E,F$ and the initial data satisfy the hypotheses of Theorem \ref{theo1}. Then for each $T>0$, there exists at most one strong solution of \eqref{0} in $\Om\times(0,T)$ in the sense of Definition \ref{def2.1}.
\end{lem}
\proof
Let $(u,\Theta,g)$ and $(\tu,\tT,\tg)$ be two strong solutions of \eqref{0} in $\Om\times(0,T)$, fix $T_0\in(0,T)$ and set
\bea{2.5}
d(x,t):=\begin{cases} u(x,t)-\tu(x,t),&x\in\ov\Om,\ t\in[0,T_0), \\ 0,&x\in\ov\Om,\ t<0, \end{cases}
\eea
and
\bea{2.6}
\delta(x,t):=\Theta(x,t)-\tT(x,t),\qquad x\in\ov\Om,\ t\in[0,T_0),
\eea
as well as
\bea{2.7}
\varrho(x,t):=g(x,t)-\tg(x,t),\qquad x\in\ov\Om,\ t\in[0,T_0).
\eea
By Definition \ref{def2.1}, all the functions $d$, $d_{t}$, $d_{tt}$, $\delta$ and $\varrho$ are continuous on their respective domains, with $d_{tt}(\cdot,0)=0$ on $\ov\Om$ and also $\delta(\cdot,0)=0$ and $\varrho(\cdot,0)=0$ on $\Om$. Since $T_0<T$, we may fix $c_i=c_i(T_0)>0$, $i\in\{1,2,\dots,13\}$, such that in line with Definition \ref{def2.1},
\bea{2.8}
\|u_x\|_{L^\infty(\Om)}+\|\tu_x\|_{L^\infty(\Om)}\le c_1\quad\text{and} \quad\|u_{xt}\|_{L^\infty(\Om)}+\|\tu_{xt}\|_{L^\infty(\Om)}\le c_2\qquad\text{for all }t\in(0,T_0)\nn
\eea
as well as
\bea{2.9}
\|\Theta\|_{L^\infty(\Om)}+\|\tT\|_{L^\infty(\Om)}\le c_3\quad\text{and}\quad\|g_t\|_{L^\infty(\Om)}+\|\tg_t\|_{L^\infty(\Om)}\le c_4\qquad\text{for all }t\in(0,T_0),
\eea
and that
\bea{2.10}
c_5\le D(\xi)\le c_6,\quad C(\xi)\le c_7,\quad |E(\xi)|\le c_8\quad\text{and}\quad |F(\xi)|\le c_9\qquad\text{for all }\xi\in[0,c_3]
\eea
as well as
\bea{2.11}
|D'(\xi)|\le c_{10},\quad|C'(\xi)|\le c_{11},\quad|E'(\xi)|\le c_{12}\quad\text{and}\quad|F'(\xi)|\le c_{13}\qquad\text{for all }\xi\in[0,c_3].
\eea
By the mean value theorem, \eqref{2.9} and \eqref{2.11} imply that
\bea{re2.10}
|D(\Theta)-D(\tT)|\le c_{10}|\delta|\quad\text{and}\quad|C(\Theta)-C(\tT)|\le c_{11}|\delta|\qquad\text{in }\Om\times(0,T_0),
\eea
and that
\bea{re2.11}
|E(\Theta)-E(\tT)|\le c_{12}|\delta|\quad\text{and}\quad|F(\Theta)-F(\tT)|\le c_{13}|\delta|\qquad\text{in }\Om\times(0,T_0).
\eea
We subtract from \eqref{2.4}, as satisfied by $(u,\Theta)$, the corresponding identity satisfied by $(\tu,\tT)$, and test the resulting equation for $d$ with suitable time-averaged versions of $d_{tt}$ itself, obtained by inserting Steklov averages of $d_{tt}$ and letting the averaging parameter tend to zero, as is standard in this type of argument. This yields, for every $t_0\in(0,T_0)$, an inequality of the form
\bea{2.12}
&&\hspace{-10mm}
\frac{\tau}{2}\io d_{tt}^2(\cdot,t_0)+\frac{1}{2}\io D(\Theta(\cdot,t_0))d_{xt}^2(\cdot,t_0)+\io C(\Theta(\cdot,t_0))d_x(\cdot,t_0)d_{xt}(\cdot,t_0)\nn\\
&\le& c_{14}\int_0^{t_0}\io d_{tt}^2+c_{14}\int_0^{t_0}\io d_{xt}^2+c_{14}\int_0^{t_0}\io d_x^2+c_{14}\int_0^{t_0}\io \delta_x^2+c_{14}\int_0^{t_0}\io \delta^2;
\eea
we refer to \cite[Proof of Lemma 2.2, inequality (2.30)]{CW26} for the details of this testing procedure, which carries over to the present setting after the substitution described in the Introduction, with $c_{14}=c_{14}(T_0)>0$ arising, via Hölder's and Young's inequalities, from the bounds $c_1,\dots, c_{13}$ fixed above together with corresponding bounds for $\tu_{xx}$, $\tu_{xxt}$, $\Theta_x$, $\tT_x$, $\Theta_t$, and from the embedding $W^{1,2}(\Om)\hookrightarrow L^\infty(\Om)$. Here, in line with the continuity of $d_x$ and $d_{xt}$, Young's inequality gives
\bea{2.13}
\frac{1}{2}\frac{d}{dt}\io d_x^2=\io d_x d_{xt}\le \io d_{xt}^2+\frac{1}{4}\io d_x^2\qquad\text{for all }t\in(0,T_0),
\eea
and, since
\bea{2.14}
\tau\varrho_t&=&-\varrho+D(\Theta)u_{xt}+C(\Theta)u_x-D(\tT)\tu_{xt}-C(\tT)\tu_x\nn\\
&=&-\varrho+\big\{D(\Theta)-D(\tT)\big\}\tu_{xt}+D(\Theta)d_{xt}+\big\{C(\Theta)-C(\tT)\big\}\tu_x+C(\Theta)d_x
\eea
holds in the classical sense in $\Om\times(0,T_0)$, multiplying by $\varrho$ and integrating over $\Om$ yields, upon using \eqref{re2.10} together with Young's inequality,
\bea{2.15}
\frac{\tau}{2}\frac{d}{dt}\io \varrho^2+\io \varrho^2&=&\io \big\{D(\Theta)-D(\tT)\big\}\tu_{xt}\cdot\varrho+\io D(\Theta)d_{xt}\cdot\varrho\nn\\
&&+\io\big\{C(\Theta)-C(\tT)\big\}\tu_x\cdot\varrho+\io C(\Theta)d_x\cdot\varrho\nn\\
&\le&c_2c_{10}\io |\delta|\cdot|\varrho|+c_6\io |d_{xt}|\cdot|\varrho|+c_1c_{11}\io|\delta|\cdot|\varrho|+c_7\io|d_x|\cdot|\varrho|\nn\\
&\le& c_{15}\io \delta^2+c_{15}\io \varrho^2+\io d_{xt}^2+c_{15}\io d_x^2\qquad\text{for all }t\in(0,T_0)
\eea
with $c_{15}:=\frac{c_2c_{10}+c_1c_{11}+c_7}{2}+\frac{c_6^2}{4}$. Moreover, in the classical sense in $\Om\times(0,T_0)$ we have
\bea{2.16}
\delta_t&=&\delta_{xx}+|E(\Theta)u_{xt}+F(\Theta)g_t|^2-|E(\tT)\tu_{xt}+F(\tT)\tg_t|^2\nn\\
&=&\delta_{xx}+\big\{E(\Theta)u_{xt}-E(\tT)\tu_{xt}+F(\Theta)g_t-F(\tT)\tg_t\big\}\nn\\
&&\hspace{10mm}\cdot\, \big\{E(\Theta)u_{xt}+E(\tT)\tu_{xt}+F(\Theta)g_t+F(\tT)\tg_t\big\},
\eea
where
\bea{2.17}
&&\hspace{-10mm}
E(\Theta)u_{xt}-E(\tT)\tu_{xt}+F(\Theta)g_t-F(\tT)\tg_t\nn\\
&=&\big\{E(\Theta)-E(\tT)\big\}\tu_{xt}+E(\Theta)d_{xt}+\big\{F(\Theta)-F(\tT)\big\}\tg_{t}+F(\Theta)\varrho_{t}\qquad\text{in }\Om\times(0,T_0).
\eea
Since
$E(\Theta)u_{xt}+E(\tT)\tu_{xt}+F(\Theta)g_t+F(\tT)\tg_t\le c_{16}\equiv c_{16}(T_0):= c_2c_8+c_4c_9$ in $\Om\times(0,T_0)$,
inserting \eqref{2.17} and \eqref{2.14} into \eqref{2.16} and using \eqref{re2.10}, \eqref{re2.11} together with Young's inequality therefore yields
\bea{nn}
\frac{1}{2}\frac{d}{dt}\io \delta^2+\io \delta_x^2&\le& c_{16}\io\big| \big\{E(\Theta)-E(\tT)\big\}\tu_{xt}\cdot\delta\big|+c_{16}\io \big|E(\Theta)d_{xt}\cdot\delta\big|\nn\\
&&+c_{16}\io \big|\big\{F(\Theta)-F(\tT)\big\}\tg_{t}\cdot\delta\big|+\frac{c_{16}}{\tau}\io \big|F(\Theta)\varrho\cdot\delta\big|\nn\\
&&+\frac{c_{16}}{\tau}\io \big|F(\Theta)\big\{D(\Theta)-D(\tT)\big\}\tu_{xt}\cdot\delta\big|+\frac{c_{16}}{\tau}\io \big|F(\Theta)D(\Theta)d_{xt}\cdot\delta\big|\nn\\
&&+\frac{c_{16}}{\tau}\io \big|F(\Theta)\big\{C(\Theta)-C(\tT)\big\}\tu_x\cdot\delta\big|+\frac{c_{16}}{\tau}\io \big|F(\Theta)C(\Theta)d_x\cdot\delta\big|\nn\\
&\le& c_2c_{12}c_{16}\io \delta^2+c_8c_{16}\io |d_{xt}|\cdot|\delta|+c_4c_{13}c_{16}\io \delta^2+\frac{c_9c_{16}}{\tau}\io |\varrho|\cdot|\delta|\nn\\
&&+\frac{c_2c_9c_{10}c_{16}}{\tau}\io\delta^2+\frac{c_6c_9c_{16}}{\tau}\io |d_{xt}|\cdot|\delta|+\frac{c_1c_9c_{11}c_{16}}{\tau}\io \delta^2+\frac{c_7c_9c_{16}}{\tau}\io |d_x|\cdot|\delta|\nn\\
&\le& c_{17}\io \delta^2+\io d_{xt}^2+c_{17}\io\varrho^2+c_{17}\io d_x^2\qquad\text{for all }t\in(0,T_0),\nn
\eea
with $c_{17}:=c_2c_{12}c_{16}+c_4c_{13}c_{16}+\frac{c_8^2c_{16}^2}{2}+\frac{c_9c_{16}+c_7c_9c_{16}}{2\tau}+\frac{c_6^2c_9^2c_{16}^2}{2\tau^2}+\frac{c_2c_9c_{10}c_{16}+c_1c_9c_{11}c_{16}}{\tau}$.
On integrating this, \eqref{2.13} and \eqref{2.15} in time and combining the result with \eqref{2.12}, we thus infer that 
\bea{nn}
y(t):=\frac{\tau}{2}\io d_{tt}^2+\frac{1}{2}\io D(\Theta)d_{xt}^2+\io C(\Theta)d_xd_{xt}+\frac{B_1}{2}\io d_x^2+\frac{B_2}{2}\io \delta^2+\frac{\tau}{2}\io \varrho^2,\qquad t\in[0,T_0],\nn
\eea
with
\bea{2.18}
B_1:=\frac{4c_7^2}{c_5}\qquad\text{and}\qquad B_2:=c_{14},\nn
\eea
satisfies
\bea{nn}
y(t_0)+B_2\int_0^{t_0}\io\delta_x^2&\le& c_{14}\int_0^{t_0}\io d_{tt}^2+c_{14}\int_0^{t_0}\io d_{xt}^2+c_{14}\int_0^{t_0}\io d_x^2+c_{14}\int_0^{t_0}\io \delta_x^2\nn\\
&&+c_{14}\int_0^{t_0}\io\delta^2+B_1\int_0^{t_0}\io d_{xt}^2+\frac{B_1}{4}\int_0^{t_0}\io d_x^2\nn\\
&&+B_2c_{17}\int_0^{t_0}\io \delta^2+B_2\int_0^{t_0}\io d_{xt}^2+B_2c_{17}\int_0^{t_0}\io\varrho^2+B_2c_{17}\int_0^{t_0}\io d_x^2\nn\\
&&+c_{15}\int_0^{t_0}\io\delta^2+c_{15}\int_0^{t_0}\io\varrho^2+\int_0^{t_0}\io d_{xt}^2+c_{15}\int_0^{t_0}\io d_x^2\nn
\eea
for all $t\in(0,T_0)$ and hence
\bea{2.19}
y(t_0)\le c_{18}\int_0^{t_0}\io d_{tt}^2+c_{18}\int_0^{t_0}\io d_{xt}^2+c_{18}\int_0^{t_0}\io d_x^2+c_{18}\int_0^{t_0}\io\delta^2+c_{18}\int_0^{t_0}\io\varrho^2
\eea
for all $t\in(0,T_0)$ with $c_{18}:=\max\{c_{14}+B_1+B_2+1,\ c_{14}+\frac{B_1}{4}+B_2c_{17}+c_{15}\}$. As, on the other hand, \eqref{2.10} together with Young’s inequality ensures that
\bea{2.20}
y(t)&\ge&\frac{\tau}{2}\io d_{tt}^2+\frac{c_5}{2}\io d_{xt}^2-c_7\io |d_x|\cdot|d_{xt}|+\frac{B_1}{2}\io d_x^2+\frac{B_2}{2}\io \delta^2+\frac{\tau}{2}\io\varrho^2\nn\\
&\ge&\frac{\tau}{2}\io d_{tt}^2+\frac{c_5}{4}\io d_{xt}^2-\frac{c_7^2}{c_5}\io d_x^2+\frac{B_1}{2}\io d_x^2+\frac{B_2}{2}\io \delta^2+\frac{\tau}{2}\io\varrho^2\nn\\
&=&\frac{\tau}{2}\io d_{tt}^2+\frac{c_5}{4}\io d_{xt}^2+\frac{B_1}{4}\io d_x^2+\frac{B_2}{2}\io \delta^2+\frac{\tau}{2}\io\varrho^2\qquad\text{for all }t\in(0,T_0),
\eea
from \eqref{2.19} we particularly obtain that
$$y(t_0)\le c_{19}\int_0^{t_0}y(t)dt\qquad\text{for all }t\in(0,T_0),$$
where $c_{19}:=c_{18}\cdot\max\{\frac{2}{\tau},\frac{4}{c_5},\frac{4}{B_1},\frac{2}{B_2},\frac{2}{\tau}\}$. Since $y$ is continuous on $[0,T_0]$ and, by the lower bound corresponding to \eqref{2.20}, nonnegative there, a Grönwall lemma implies that $y(t_0) = 0$ for all $t_0\in(0,T_0)$, whence again relying on \eqref{2.20} we may conclude that $d_{tt}\equiv 0$, $\delta\equiv 0$ and $\varrho\equiv 0$ in $\Om\times(0,T_0)$. Since $d(x,0)=d_t(x,0)=0$ for all $x\in\Om$ by Definition \eqref{def2.1}, and since $T_0\in(0,T)$ was arbitrary, this means that, indeed, $u-\tu\equiv 0$, $\Theta-\tT\equiv 0$ and $g-\tg\equiv 0$ in $\Om\times(0,T)$.
\qed
\subsection{Approximate problems and a priori estimates}\label{sec2.2}
We recall the strategy of regularizing the original problem by a family of parabolic approximations, as employed in \cite[Section 3]{CW26}. In order to address the existence claim in a conveniently simple setting, essentially equivalent to that covered by Theorem \ref{theo1}, we first restrict our attention to initial data satisfying, beyond the requirement that $u_0, u_{0t}\in W_N^{2,2}(\Om)$, $u_{0tt},g_0\in W^{1,2}(\Om)$ and that $\Theta_0\in W_N^{2,\infty}(\Om)$, the normalization
\bea{---}
\io u_0=\io u_{0t}=\io u_{0tt}=0,
\eea
and only remove this restriction at the very end, in the proof of Theorem \ref{theo1} itself.\abs Specifically, here and below we assume that $\Om\subset\mathbb{R}$ is a fixed open bounded interval, that $\tau>0$, and that $D$, $C$, $E$ and $F$ comply with \eqref{1.10}. By standard smoothing procedures, we fix families
\bea{2.30}
(D_\eps)_{\eps\in(0,1)},(C_\eps)_{\eps\in(0,1)},(E_\eps)_{\eps\in(0,1)},(F_\eps)_{\eps\in(0,1)}\subset C^\infty([0,\infty))\text{ with }D_\eps,C_\eps>0\text{ on }[0,\infty)
\eea
for all $\eps\in(0,1)$, such that
\bea{2.31}
D_\eps\to D\quad\text{and}\quad C_\eps\to C\qquad\text{in }C_{loc}^2([0,\infty))\text{ as }\eps\searrow 0,
\eea
\bea{2.32}
E_\eps\to E\quad\text{and}\quad F_\eps\to F\qquad\text{in }C_{loc}^1([0,\infty))\text{ as }\eps\searrow 0,
\eea
as well as
\be{2.33}
	\lball
	(u_{0\eps})_{\eps\in(0,1)}\subset C^\infty(\ov\Om),\\[1mm]
	(v_{0\eps})_{\eps\in(0,1)}\subset C^\infty(\ov\Om),\\[1mm]
	(w_{0\eps})_{\eps\in(0,1)}\subset C^\infty(\ov\Om),\\[1mm]
	(\Theta_{0\eps})_{\eps\in(0,1)}\subset C^\infty(\ov\Om)\quad\text{with }\Theta_{0\eps}\ge 0\text{ for all }\eps\in(0,1)\qquad\text{and}\\[1mm]
	(g_{0\eps})_{\eps\in(0,1)}\subset C^\infty(\ov\Om)\quad\text{with }g_{0\eps}=0\text{ on }\partial\Om\text{ for all }\eps\in(0,1),
	\ear
\ee
that $u_{0\eps x}$, $v_{0\eps x}$, $w_{0\eps x}$ and $\Theta_{0\eps x}$ have compact support in $\Om$ with
\bea{2.34}
\io u_{0\eps}=\io v_{0\eps}=\io w_{0\eps}=0\qquad\text{for all }\eps\in(0,1),
\eea
and such that as $\eps\searrow 0$ we have
\be{2.35}
	\lball
	u_{0\eps}\to u_0\qquad\text{in }W^{2,2}(\Om),\\[1mm]
	v_{0\eps}\to u_{0t}\qquad\text{in }W^{2,2}(\Om),\\[1mm]
	w_{0\eps}\to u_{0tt}\qquad\text{in }W^{1,2}(\Om),\\[1mm]
	\Theta_{0\eps}\stackrel{\ast}{\rightharpoonup} \Theta_0\qquad\text{in }W^{2,\infty}(\Om)\qquad\text{and}\\[1mm]
	g_{0\eps}\to g_0\qquad\text{in }W^{1,2}(\Om),
	\ear
\ee
as well as
\bea{2.36}
\sqrt{\eps}u_{0\eps xxx}\to 0\qquad\text{in }L^2(\Om).
\eea
For $\eps\in(0,1)$, standard arguments from local existence theory for parabolic cross-diffusion systems, applied to the regularization of \eqref{0} given by
\begin{eqnarray}
\left\{ \begin{array}{lcll}
\tau w_{\eps t}\hspace{-2mm}&=&\hspace{-2mm}\eps\tau w_{\eps xx}+(D_\eps(\Theta_\eps)v_{\eps x})_x+(C_\eps(\Theta_\eps)u_{\eps x})_x-w_\eps, & x\in\Om,\ t>0,\\[1mm]
v_{\eps t}\hspace{-2mm}&=&\hspace{-2mm}\eps v_{\eps xx}+w_\eps,& x\in\Om,\ t>0,\\[1mm]
u_{\eps t}\hspace{-2mm}&=&\hspace{-2mm}\eps u_{\eps xx}+v_\eps,& x\in\Om,\ t>0,\\[1mm]
\Theta_{\eps t}\hspace{-2mm}&=&\hspace{-2mm}\Theta_{\eps xx}+|E_\eps(\Theta_\eps)v_{\eps x}+F_\eps(\Theta_\eps)h_{\eps}|^2,& x\in\Om,\ t>0,\\[1mm]
\tau g_{\eps t}\hspace{-2mm}&=&\hspace{-2mm}\eps\tau g_{\eps xx}+D_\eps(\Theta_\eps)v_{\eps x}+C_\eps(\Theta_\eps)u_{\eps x}-g_\eps,& x\in\Om,\ t>0,\\[1mm]
\multicolumn{3}{l}{w_{\varepsilon x}=v_{\varepsilon x}=u_{\varepsilon x}=\Theta_{\varepsilon x}=g_\varepsilon=0,} & x\in\partial\Omega,\ t>0,\\[1mm]
\multicolumn{3}{l}{(w_\eps,v_\eps,u_\eps,\Theta_\eps,g_\eps)(x,0)=(w_{0\eps},v_{0\eps},u_{0\eps},\Theta_{0\eps},g_{0\eps})(x),} & x\in\Om,
\end{array} \right. \label{approx}
\end{eqnarray}
with 
\bea{2.36.1}
h_\eps:=\frac{1}{\tau}\big\{D_\eps(\Theta_\eps)v_{\eps x}+C_\eps(\Theta_\eps)u_{\eps x}-g_\eps\big\},\qquad x\in\Om,\ t>0,
\eea
standing in for $g_{\eps t}$ in accordance with the limit $\eps\searrow 0$ considered in Section \ref{sec2.3}, yield the following.
\begin{lem}\label{lem2.3}
Let $\eps\in(0,1)$. Then there exist $\tme\in(0,\infty]$ and
\bea{2.37}
w_\eps,v_\eps,u_\eps,\Theta_\eps,g_\eps\in C^{2,1}(\ov\Om\times[0,\tme))\cap C^\infty(\ov\Om\times(0,\tme)),
\eea
%as well as
%\bea{2.38}
%g_\eps\in C^1([0,\tme);C^0(\ov\Om))\cap C^\infty((0,\tme);C^0(\ov\Om))
%\eea
such that $\Theta_\eps\ge 0$ in $\Om\times(0,\tme)$, that $(w_\eps,v_\eps,u_\eps,\Theta_\eps,g_\eps)$ solves \eqref{approx} classically in $\Om\times(0,\tme)$, and that
\bea{2.39}
if\ \tme<\infty,\qquad then\hspace{10cm}\nn\\
\limsup_{t\nearrow\tme}\Big\{\|w_\eps(\cdot,t)\|_{L^\infty(\Om)}+\|v_\eps(\cdot,t)\|_{W^{1,2}(\Om)}+\|u_\eps(\cdot,t)\|_{W^{1,2}(\Om)}\hspace{3cm}\nn\\
+\|\Theta_\eps(\cdot,t)\|_{L^\infty(\Om)}+\|g_\eps(\cdot,t)\|_{L^\infty(\Om)}\Big\}=\infty.\hspace{1.5cm}
\eea
Moreover,
\bea{2.40}
\io w_\eps(\cdot,t)=\io v_\eps(\cdot,t)=\io u_\eps(\cdot,t)=0\qquad\text{for all }t\in(0,\tme).
\eea
\end{lem}
\proof
Following the approach of \cite[Proof of Lemma 3.1]{CW26}, we fix $\eps\in(0,1)$, and choose extensions $D_{0\eps},C_{0\eps},E_{0\eps},F_{0\eps}\in C^\infty(\mathbb{R})$ such that $D_{0\eps}|_{[0,\infty)}=D_\eps$, $C_{0\eps}|_{[0,\infty)}=C_\eps$, $E_{0\eps}|_{[0,\infty)}=E_\eps$ and $F_{0\eps}|_{[0,\infty)}=F_\eps$. The Neumann heat semigroup $(e^{t\Del})_{t\ge 0}$ on $\Om$ satisfies well-known smoothing estimates (\cite{Winkler10}): there are $c_1,c_2,c_3>0$ such that for each $t\in(0,1)$,
\bea{2.41}
\|e^{\eps t\Del}\varphi\|_{W^{1,2}(\Om)}\le c_1 t^{-\frac{1}{2}}\|\varphi\|_{L^\infty(\Om)}\qquad\text{for all }\varphi\in C^0(\ov\Om)
\eea 
and
\bea{2.42}
\|e^{t\Del}\varphi\|_{L^\infty(\Om)}\le c_2 t^{-\frac{1}{2}}\|\varphi\|_{L^1(\Om)}\qquad\text{for all }\varphi\in L^1(\Om),
\eea
as well as
\bea{2.43}
\|e^{\eps t\Del}\partial_x\varphi\|_{L^\infty(\Om)}\le c_3 t^{-\frac{3}{4}}\|\varphi\|_{L^2(\Om)}\qquad\text{for all }\varphi\in C_0^\infty(\Om).
\eea
We moreover note that the maximum principle ensures that for any $t>0$,
\bea{2.44}
\|e^{t\Del}\varphi\|_{W^{1,2}(\Om)}\le \|\varphi\|_{W^{1,2}(\Om)}\qquad\text{for all }\varphi\in W^{1,2}(\Om)
\eea
and
\bea{2.45}
\|e^{t\Del}\varphi\|_{L^\infty(\Om)}\le\|\varphi\|_{L^\infty(\Om)}\qquad\text{for all }\varphi\in C^0(\ov\Om).
\eea
We furthermore fix a constant $c_4>0$ governing the smoothing action of the Dirichlet heat semigroup: denoting by $\Del_D$ the realization of the Laplacian in $\Om$ under the homogeneous Dirichlet condition shows that
\bea{2.45.1}
\|e^{\eps t\Del_D}\varphi\|_{L^\infty(\Om)}\le c_4 t^{-\frac{1}{4}}\|\varphi\|_{L^{2}(\Om)}\qquad\text{for all }\varphi\in L^{2}(\Om)\text{ and }t\in(0,1),
\eea
which follows by the same reasoning as the corresponding estimate for the Neumann heat semigroup established in \cite{Winkler10}. By the maximum principle, this semigroup is moreover non-expansive on $C^0(\ov\Om)$, in analogy to \eqref{2.45}; that is,
\bea{2.45.2}
\|e^{t\Del_D}\varphi\|_{L^\infty(\Om)}\le\|\varphi\|_{L^\infty(\Om)}\qquad\text{for all }\varphi\in C^0(\ov\Om).
\eea
Let $X_0:=C^0(\ov\Om)\times W^{1,2}(\Om)\times W^{1,2}(\Om)\times C^0(\ov\Om)\times C^0(\ov\Om)$, equipped with the norm
$$\big\|(\varphi_1,\varphi_2,\varphi_3,\varphi_4,\varphi_5)\big\|_{X_0}:=\max\Big\{\|\varphi_1\|_{L^\infty(\Om)},\ \|\varphi_2\|_{W^{1,2}(\Om)},\ \|\varphi_3\|_{W^{1,2}(\Om)},\ \|\varphi_4\|_{L^\infty(\Om)},\ \|\varphi_5\|_{L^\infty(\Om)}\Big\}$$
for $(\varphi_1,\varphi_2,\varphi_3,\varphi_4,\varphi_5)\in X_0$, and set
\bea{2.46}
R:=\big\|(w_{0\eps},v_{0\eps},u_{0\eps},\Theta_{0\eps},g_{0\eps})\big\|_{X_0}+1.
\eea
We may then pick $T_0=T_0(R)\in(0,1)$ small enough fulfilling
\bea{2.47}
4c_3\cdot\big\{\|D_{0\eps}\|_{L^\infty([-R,R])}+\|C_{0\eps}\|_{L^\infty([-R,R])}\big\}\cdot R T_0^\frac{1}{4}+RT_0\le \tau,\quad 2c_1RT_0^\frac{1}{2}\le 1\quad\text{and}\quad RT_0\le 1,
\eea
these are the bounds required in \cite[Proof of Lemma 3.1]{CW26} to handle $\Phi_1,\Phi_2,\Phi_3$ below, while additionally
\bea{2.48}
\frac{8c_2}{\tau^2}\|F_{0\eps}\|_{L^\infty([-R,R])}^2\cdot\big\{\|D_{0\eps}\|_{L^\infty([-R,R])}^2+\|C_{0\eps}\|_{L^\infty([-R,R])}^2+|\Om|\big\}\cdot R^2T_0^\frac{1}{2}\hspace{2cm}\nn\\
+8c_2\|E_{0\eps}\|_{L^\infty([-R,R])}^2R^2T_0^\frac{1}{2}\le 1
\eea
and
\bea{2.49}
\frac{4c_4}{3\tau}\cdot\big\{\|D_{0\eps}\|_{L^\infty([-R,R])}+\|C_{0\eps}\|_{L^\infty([-R,R])}\big\}\cdot RT_0^\frac{3}{4}+\frac{1}{\tau}RT_0\le 1.
\eea
For $T\in(0,T_0]$, consider the Banach space $X:=C^0([0,T];X_0)$ with norm $\|\varphi\|_X:=\sup_{t\in[0,T]}\|\varphi(\cdot,t)\|_{X_0}$ for $\varphi\in X$, and the closed set $S:=\{\varphi\in X\ |\ \|\varphi\|_X\le R\}$. Guided by the structure of \eqref{approx}, we introduce a mapping $\Phi=(\Phi_1,\Phi_2,\Phi_3,\Phi_4,\Phi_5):S\to X$ by
$$\big(\Phi_1(w,v,u,\Theta,g)\big)(\cdot,t):=e^{\eps t\Del}w_{0\eps}+\frac{1}{\tau}\int_0^t e^{\eps (t-s)\Del}\partial_x\big\{D_{0\eps}(\Theta)v_x+C_{0\eps}(\Theta)u_x\big\}ds-\frac{1}{\tau}\int_0^t e^{\eps(t-s)\Del}wds,$$
$$\big(\Phi_2(w,v,u,\Theta,g)\big)(\cdot,t):=e^{\eps t\Del}v_{0\eps}+\int_0^t e^{\eps(t-s)\Del}wds,$$
$$\big(\Phi_3(w,v,u,\Theta,g)\big)(\cdot,t):=e^{\eps t\Del}u_{0\eps}+\int_0^t e^{\eps(t-s)\Del}vds,$$
as well as
$$\big(\Phi_4(w,v,u,\Theta,g)\big)(\cdot,t):=e^{t\Del}\Theta_{0\eps}+\int_0^t e^{(t-s)\Del}\Big|E_{0\eps}(\Theta)v_x+\frac{F_{0\eps}(\Theta)}{\tau}\big\{D_{0\eps}(\Theta)v_x+C_{0\eps}(\Theta)u_x-g\big\}\Big|^2ds$$
and
$$\big(\Phi_5(w,v,u,\Theta,g)\big)(\cdot,t):=e^{\eps t\Del_D}g_{0\eps}+\frac{1}{\tau}\int_0^t e^{\eps(t-s)\Del_D}\big\{D_{0\eps}(\Theta)v_x+C_{0\eps}(\Theta)u_x\big\}ds-\frac{1}{\tau}\int_0^t e^{\eps(t-s)\Del_D}gds$$
for $(w,v,u,\Theta,g)\in S$ and $t\in[0,T]$. We now verify that $\Phi$ maps $S$ into itself. Since $\Phi_1$, $\Phi_2$ and $\Phi_3$ coincide, under the substitution described in the Introduction, with the corresponding operators in \cite[Proof of Lemma 3.1]{CW26}, the same computations, based on \eqref{2.41} and \eqref{2.43}-\eqref{2.45}, show that for all $t\in[0,T]$
\bea{nn}
&&\hspace{-10mm}\big\|\big(\Phi_1(w,v,u,\Theta,g)\big)(\cdot,t)\big\|_{L^\infty(\Om)}\nn\\
&&\le \|w_{0\eps}\|_{L^\infty(\Om)}+\frac{4c_3}{\tau}\cdot\big\{\|D_{0\eps}\|_{L^\infty([-R,R])}+\|C_{0\eps}\|_{L^\infty([-R,R])}\big\}\cdot RT^\frac{1}{4}+\frac{1}{\tau}RT\le R,\nn
\eea
$$\big\|\big(\Phi_2(w,v,u,\Theta,g)\big)(\cdot,t)\big\|_{W^{1,2}(\Om)}\le \|v_{0\eps}\|_{W^{1,2}(\Om)}+2c_1RT^\frac{1}{2}\le R,$$
$$\big\|\big(\Phi_3(w,v,u,\Theta,g)\big)(\cdot,t)\big\|_{W^{1,2}(\Om)}\le \|u_{0\eps}\|_{W^{1,2}(\Om)}+RT\le R,$$
where the last inequality in each line holds because of \eqref{2.46} and \eqref{2.47}. Similarly, \eqref{2.42} and \eqref{2.45} together with \eqref{2.46} and \eqref{2.48} ensure that
\bea{nn}
&&\hspace{-10mm}\big\|\big(\Phi_4(w,v,u,\Theta,g)\big)(\cdot,t)\big\|_{L^\infty(\Om)}\nn\\
&\le& \|\Theta_{0\eps}\|_{L^\infty(\Om)}+c_2\int_0^t(t-s)^{-\frac{1}{2}}\bigg\|\Big|E_{0\eps}(\Theta)v_x+\frac{F_{0\eps}(\Theta)}{\tau}\big\{D_{0\eps}(\Theta)v_x+C_{0\eps}(\Theta)u_x-g\big\}\Big|^2\bigg\|_{L^1(\Om)}ds\nn\\
&\le&\|\Theta_{0\eps}\|_{L^\infty(\Om)}+\frac{4c_2}{\tau^2}\int_0^t(t-s)^{-\frac{1}{2}}\big\|F_{0\eps}^2(\Theta)\big\{|D_{0\eps}(\Theta)v_x|^2+|C_{0\eps}(\Theta)u_x|^2+g^2\big\}\big\|_{L^1(\Om)}ds\nn\\
&&+4c_2\int_0^t(t-s)^{-\frac{1}{2}}\big\||E_{0\eps}(\Theta)v_x|^2\big\|_{L^1(\Om)}ds\nn\\
&\le& \|\Theta_{0\eps}\|_{L^\infty(\Om)}+\frac{4c_2}{\tau^2}\|F_{0\eps}\|_{L^\infty([-R,R])}^2\cdot\big\{\|D_{0\eps}\|_{L^\infty([-R,R])}^2+\|C_{0\eps}\|_{L^\infty([-R,R])}^2+|\Om|\big\}\cdot R^2\int_0^t(t-s)^{-\frac{1}{2}}ds\nn\\
&&+4c_2\|E_{0\eps}\|_{L^\infty([-R,R])}^2R^2\int_0^t(t-s)^{-\frac{1}{2}}ds\nn\\
&\le& \|\Theta_{0\eps}\|_{L^\infty(\Om)}+\frac{8c_2}{\tau^2}\|F_{0\eps}\|_{L^\infty([-R,R])}^2\cdot\big\{\|D_{0\eps}\|_{L^\infty([-R,R])}^2+\|C_{0\eps}\|_{L^\infty([-R,R])}^2+|\Om|\big\}\cdot R^2T^\frac{1}{2}\nn\\
&&+8c_2\|E_{0\eps}\|_{L^\infty([-R,R])}^2R^2T^\frac{1}{2}\nn\\
&\le& R\qquad\text{for all }t\in[0,T],\nn
\eea
and that, using \eqref{2.45.1} and \eqref{2.45.2}, together with \eqref{2.46} and \eqref{2.49},
\bea{nn}
&&\hspace{-10mm}\big\|\big(\Phi_5(w,v,u,\Theta,g)\big)(\cdot,t)\big\|_{L^\infty(\Om)}\nn\\
&\le&\|g_{0\eps}\|_{L^\infty(\Om)}+\frac{c_4}{\tau}\int_0^t(t-s)^{-\frac{1}{4}} \cdot\big\{\|D_{0\eps}(\Theta)v_x\|_{L^2(\Om)}+\|C_{0\eps}(\Theta)u_x\|_{L^2(\Om)}\big\}ds+\frac{1}{\tau}\int_0^t \|g\|_{L^\infty(\Om)}ds\nn\\
&\le&\|g_{0\eps}\|_{L^\infty(\Om)}+\frac{c_4}{\tau}\cdot\big\{\|D_{0\eps}\|_{L^\infty([-R,R])}+\|C_{0\eps}\|_{L^\infty([-R,R])}\big\}\cdot R\int_0^t(t-s)^{-\frac{1}{4}}ds+\frac{1}{\tau}RT\nn\\
&\le&\|g_{0\eps}\|_{L^\infty(\Om)}+\frac{4c_4}{3\tau}\cdot\big\{\|D_{0\eps}\|_{L^\infty([-R,R])}+\|C_{0\eps}\|_{L^\infty([-R,R])}\big\}\cdot RT^\frac{3}{4}+\frac{1}{\tau}RT\nn\\
&\le& R\qquad\text{for all }t\in[0,T],
\eea
because $T\le T_0\le 1$. Combining these five estimates, we conclude that $\Phi$ maps $S$ into itself for every $T\in(0,T_0]$. A straightforward adaptation of the same computations shows moreover that $\Phi$ is a contraction on $S$ provided $T=T(R)\in(0,T_0]$ is chosen sufficiently small. Fixing such a $T$, the Banach fixed point theorem yields  $(w_\eps,v_\eps,u_\eps,\Theta_\eps,g_\eps)\in S$ satisfying $\Phi(w_\eps,v_\eps,u_\eps,\Theta_\eps,g_\eps)=(w_\eps,v_\eps,u_\eps,\Theta_\eps,g_\eps)$, and standard parabolic regularity theory (\cite{LSU}) shows that these functions in fact belong to $C^{2,1}(\ov\Om\times[0,T])\cap C^\infty(\ov\Om\times(0,T])$, and solve the problem obtained from \eqref{approx} upon replacing $D_\eps,C_\eps,E_\eps$ and $F_\eps$ by $D_{0\eps},C_{0\eps},E_{0\eps}$ and $F_{0\eps}$. Since $\Theta_{0\eps}\ge 0$ and the source term $|E_{0\eps}(\Theta_\eps)v_{\eps x}+F_{0\eps}(\Theta_\eps)g_{\eps t}|^2$ is nonnegative in $\Om\times(0,T)$, a comparison principle warrants that $\Theta_\eps\ge 0$ throughout $\Om\times(0,T)$ and that thus we actually have $D_{0\eps}(\Theta_\eps)\equiv D_\eps(\Theta_\eps)$, $C_{0\eps}(\Theta_\eps)\equiv C_\eps(\Theta_\eps)$, $E_{0\eps}(\Theta_\eps)\equiv E_\eps(\Theta_\eps)$ and $F_{0\eps}(\Theta_\eps)\equiv F_\eps(\Theta_\eps)$ in $\Om\times(0,T)$, which identifies $(w_\eps,v_\eps,u_\eps,\Theta_\eps,g_\eps)$ as a classical solution of \eqref{approx} itself on $\Om\times(0,T)$. Since the length of $T$ constructed above depends on the initial data only through their norm in $X_0$, a standard continuation argument extends this solution to a maximal interval $\Om\times(0,\tme)$ with $\tme\in(0,\infty]$ satisfying the extensibility criterion \eqref{2.39}. The conservation properties in \eqref{2.40}, finally, result from an integration in \eqref{approx} using \eqref{2.34}.
\qed
We now derive bounds ensuring that the coefficients $D_\eps(\Theta_\eps)$, $C_\eps(\Theta_\eps)$, $E_\eps(\Theta_\eps)$, $F_\eps(\Theta_\eps)$ and their derivatives stay controlled as long as $\Theta_\eps$, $\Theta_{\eps t}$ and $g_\eps$ do.
\begin{lem}\label{lem2.4}
There exists $M>0$ such that, writing
\bea{3.24}
\Theta_{0\eps t}:=\Theta_{0\eps xx}+|E_\eps(\Theta_{0\eps})v_{0\eps x}+F_\eps(\Theta_{0\eps})h_{0\eps}|^2,\qquad\eps\in(0,1),
\eea
with
\bea{nn}
h_{0\eps}:=\frac{1}{\tau}\big\{D_\eps(\Theta_{0\eps})v_{0\eps x}+C_\eps(\Theta_{0\eps})u_{0\eps x}-g_{0\eps}\big\},\qquad\eps\in(0,1),\nn
\eea
we have
\bea{3.25}
\|\Theta_{0\eps}\|_{L^\infty(\Om)}+\|\Theta_{0\eps t}\|_{L^\infty(\Om)}+\|g_{0\eps}\|_{L^\infty(\Om)}\le\frac{M}{12}\qquad\text{for all }\eps\in(0,1).
\eea
Moreover, there exist $\eps_\star\in(0,1)$ and $k_i>0$, $i\in\{1,\dots,12\}$, such that if $\eps\in(0,\eps_\star)$ and $T\in(0,\tme)$ are such that
\bea{3.26}
\|\Theta_\eps(\cdot,t)\|_{L^\infty(\Om)}+\|\Theta_{\eps t}(\cdot,t)\|_{L^\infty(\Om)}+\|g_\eps(\cdot,t)\|_{L^\infty(\Om)}\le M\qquad\text{for all }t\in(0,T),
\eea
then
\bea{3.27}
k_1\le D_\eps(\Theta_\eps)\le k_2,\quad k_3\le C_\eps(\Theta_\eps)\le k_4,\quad |E_\eps(\Theta_\eps)|\le k_5\quad\text{and}\quad |F_\eps(\Theta_\eps)|\le k_6\quad\text{in }\Om\times(0,T)
\eea
and
\bea{3.28}
|D'_\eps(\Theta_\eps)|\le k_7,\quad |C'_\eps(\Theta_\eps)|\le k_8,\quad |D''_\eps(\Theta_\eps)|\le k_9\quad\text{and}\quad |C''_\eps(\Theta_\eps)|\le k_{10}\quad\text{in }\Om\times(0,T)
\eea
as well as
\bea{3.28.1}
|E'_\eps(\Theta_\eps)|\le k_{11}\quad\text{and}\quad |F'_\eps(\Theta_\eps)|\le k_{12}\quad\text{in }\Om\times(0,T).
\eea
\end{lem}
\proof
The inequality \eqref{3.25} follows immediately from \eqref{2.31}, \eqref{2.32} and \eqref{2.35}, while \eqref{3.27}-\eqref{3.28.1} result directly from the convergence properties in \eqref{2.31} and \eqref{2.32} and the positivity of $D$ and $C$ on $[0,\infty)$.
\qed
We next bound the spatial derivatives of $\Theta_\eps$ in terms of those of $v_\eps$ and $u_\eps$.
\begin{lem}\label{lem2.5}
Let $M$ and $\eps_\star$ be as in Lemma \ref{lem2.4}. Then there exists $k_{13}>0$ such that if $\eps\in(0,\eps_\star)$ and $T\in(0,\tme)$ are such that \eqref{3.26} is satisfied, then 
\bea{3.29}
\|\Theta_{\eps xx}(\cdot,t)\|_{L^\infty(\Om)}+\|\Theta_{\eps x}(\cdot,t)\|_{L^\infty(\Om)}^2\le k_{13}\|v_{\eps x}(\cdot,t)\|_{L^\infty(\Om)}^2+k_{13}\|u_{\eps x}(\cdot,t)\|_{L^\infty(\Om)}^2+k_{13}
\eea
for all $t\in(0,T)$.
\end{lem}
\proof
By the fourth equation in \eqref{approx}, with $h_\eps$ as in \eqref{2.36.1}, and using \eqref{3.26} and \eqref{3.28},
\bea{nn}
|\Theta_{\eps xx}|&=&\bigg|\Theta_{\eps t}-\Big|E_\eps(\Theta_\eps)v_{\eps x}+\frac{F_\eps(\Theta_\eps)}{\tau}\big\{D_\eps(\Theta_\eps)v_{\eps x}+C_\eps(\Theta_\eps)u_{\eps x}-g_\eps\big\}\Big|^2\bigg|\nn\\
&\le& |\Theta_{\eps t}|+4|E_\eps(\Theta_\eps)v_{\eps x}|^2+\frac{4F_\eps^2(\Theta_\eps)}{\tau^2}\big\{|D_\eps(\Theta_\eps)v_{\eps x}|^2+|C_\eps(\Theta_\eps)u_{\eps x}|^2+g_\eps^2\big\}\nn\\
&\le& M+4k_5^2v_{\eps x}^2+\frac{4k_6^2}{\tau^2}\big\{k_2^2v_{\eps x}^2+k_4^2u_{\eps x}^2 +M^2\big\}\qquad\text{in }\Om\times(0,T).\nn
\eea
Combined with the Gagliardo-Nirenberg inequality
$$\|\varphi_x\|_{L^2(\Om)}^2\le c_1\|\varphi_{xx}\|_{L^\infty(\Om)}\|\varphi\|_{L^\infty(\Om)}\qquad\text{for all }\varphi\in C^2(\ov\Om),$$
valid for some $c_1>0$, this yields \eqref{3.29} with $k_{13}:=(1+c_1M)\cdot\max\big\{M+\frac{4k_6^2M^2}{\tau^2},4k_5^2+\frac{4k_2^2k_6^2}{\tau^2},\frac{4k_4^2k_6^2}{\tau^2}\big\}$.
\qed
As in \cite{CW26}, we now bound $w_{\eps x}$, $v_{\eps xx}$ and $u_{\eps xx}$, in three steps that mirror the structure of the argument there:
\begin{lem}\label{lem2.6a}
There exists $k_{14}>0$ such that if $\eps_\star$ and $(k_i)_{i\in\{1,\dots,12\}}$ as in Lemma \ref{lem2.4}, we let $B:=\frac{4k_4^2}{k_1\tau}$ and
\bea{re2.53}
y_\eps(t):=\frac{1}{2}\io w_{\eps x}^2+\frac{1}{2\tau}\io D_\eps(\Theta_\eps)v_{\eps xx}^2+\frac{1}{\tau}\io C_\eps(\Theta_\eps)u_{\eps xx}v_{\eps xx}+\frac{B}{2}\io u_{\eps xx}^2+\frac{\eps}{\tau}\io C_\eps(\Theta_\eps)u_{\eps xxx}^2\ \
\eea
for $t\in(0,\tme)$ and $\eps\in(0,1)$, then whenever $\eps\in(0,\eps_\star)$ and $T\in(0,\tme)$ are such that \eqref{3.26} holds with $M$ taken from Lemma \ref{lem2.4}, it follows that
\bea{re2.54}
k_{14}y_\eps(t)\ge \io w_{\eps x}^2+\io v_{\eps xx}^2+\io u_{\eps xx}^2+\eps\io u_{\eps xxx}^2\qquad\text{for all }t\in(0,T).
\eea
\end{lem}
\proof
The quantity $y_\eps$ arises exactly as in \cite[Lemma 3.7]{CW26}, by testing the first three equations of \eqref{approx} against $w_{\eps x}$, $v_{\eps xx}$ and $u_{\eps xx}$, respectively, as carried out in detail in \cite[Lemma 3.4 and Lemma 3.6]{CW26}; since these three equations coincide with those in \cite{CW26} under the substitution described in the Introduction, the resulting coercivity bound, obtained there by a Young's inequality argument exploiting the bounds from Lemma \ref{lem2.4}.
\qed
\begin{lem}\label{lem2.6b}
With $M$ and $\eps_\star$ taken from Lemma \ref{lem2.4}, one can find $k_{15}>0$ with the property that if $\eps\in(0,\eps_\star)$ and $T\in(0,\tme)$ are such that \eqref{3.26} holds, then the functions in \eqref{re2.53} satisfies
\bea{re2.55}
y_\eps'(t)+k_{1}\eps\io v_{\eps xxx}^2\le k_{15}y_\eps^2(t)+k_{15}y_\eps(t)\qquad\text{for all }t\in(0,T).
\eea
\end{lem}
\proof
Differentiating $y_\eps$ along \eqref{approx} and testing as in \cite[Lemma 3.8]{CW26} produces several terms, all of which are controlled exactly as there, under the substitution described in the Introduction, except for one step: the control of $\Theta_{\eps xx}$ and $\Theta_{\eps x}$. By a Sobolev inequality, we may fix $c_2>0$ such that
$$\|\varphi\|_{L^\infty(\Om)}\le c_2\|\varphi_x\|_{L^2(\Om)}\qquad\text{for all }\varphi\in W_0^{1,2}(\Om),$$
so that, applying this to $v_{\eps x}$ and $u_{\eps x}$ and invoking Lemma \ref{lem2.5}, with $k_{13}$ as provided there, we obtain
\bea{3.45}
\|\Theta_{\eps xx}\|_{L^\infty(\Om)}+\|\Theta_{\eps x}\|_{L^\infty(\Om)}^2\le k_{13}c_1^2\io v_{\eps xx}^2+k_{13}c_{1}^2\io u_{\eps xx}^2+k_{13}\le c_2 y_\eps(t)+c_2
\eea 
for all $t\in(0,T)$ with $c_2:=\max\{k_{13}k_{14}c_1^2,k_{13}\}$, with $k_{14}$ as in Lemma \ref{lem2.6a}. This bound takes the place of \cite[(3.45)]{CW26} in the proof of \cite[Lemma 3.8]{CW26}; the only difference is the additional term $u_{\eps x}^2$, which is controlled by the same embedding and by Lemma \ref{lem2.6a} in exactly the same way as $v_{\eps x}^2$. With this modification in place, the remaining computation of \cite[Lemma 3.8]{CW26} applies unchanged and yields \eqref{re2.55}.
\qed
\begin{lem}\label{lem2.6}
Let $M$ and $\eps_\star$ be as in Lemma \ref{lem2.4}. Then there exist $\eps_{\star\star}\in(0,\eps_\star)$, $k_{16}>0$ and $T_0>0$ such that whenever $\eps\in(0,\eps_{\star\star})$ and $T\in(0,\tme)$ are such that \eqref{3.26} holds, we have
\bea{3.58}
\io w_{\eps x}^2(\cdot,t)\le k_{16},\quad\io v_{\eps xx}^2(\cdot,t)\le k_{16},\quad\text{and}\quad\io u_{\eps xx}^2(\cdot,t)\le k_{16}\qquad\text{for all }t\in(0,T)\cap(0,T_0)\ \
\eea
as well as
\bea{3.61}
\eps\int_0^t\io v_{\eps xxx}^2\le k_{16}\qquad\text{for all }t\in(0,T)\cap(0,T_0).
\eea
\end{lem}
\proof
By \eqref{2.35} and \eqref{2.36}, $y_{\eps}(0)$ remains bounded, uniformly in $\eps$, as $\eps\searrow 0$. Since the differential inequality from Lemma \ref{lem2.6b} is of Bernoulli type, a standard ODE comparison argument shows that $y_\eps$ itself remains bounded on some short, $\eps-$independent time interval $(0,T_0)$; via the coercivity bound from Lemma \ref{lem2.6a}, this yields \eqref{3.58} and \eqref{3.61}. This is precisely the argument underlying \cite[Lemma 3.10]{CW26}.
\qed
Complementing the bounds obtained so far, we now also derive one for $u_{\eps xxx}$, $g_{\eps x}$ and $g_{\eps xx}$.
\begin{lem}\label{lem2.8}
Let $M$, $\eps_{\star\star}$ and $T_0$ be as in Lemma \ref{lem2.4} and Lemma \ref{lem2.6}. Then there exists $k_{17}>0$ such that whenever $\eps\in(0,\eps_{\star\star})$ and $T\in(0,\tme)$ are such that \eqref{3.26} holds, we have
\bea{2.55.00}
\io g_{\eps x}^2(\cdot,t)\le k_{17}\qquad\text{for all }t\in(0,T)\cap(0,T_0),
\eea
and
\bea{2.55.0}
\eps\int_0^t\io u_{\eps xxx}^2\le k_{17}\qquad\text{for all }t\in(0,T)\cap(0,T_0),
\eea
as well as
\bea{2.55}
\eps\int_0^t\io g_{\eps xx}^2\le k_{17}\qquad\text{for all }t\in(0,T)\cap(0,T_0).
\eea
\end{lem}
\proof
By the continuity of the embedding $W^{1,2}(\Om)\hookrightarrow L^\infty(\Om)$, we fix $c_1>0$ such that
$$\|\varphi\|_{L^\infty(\Om)}\le c_1\|\varphi_x\|_{L^2(\Om)}\qquad\text{for all }\varphi\in W_0^{1,2}(\Om),$$
and let $(k_i)_{i\in\{1,\dots,12\}}$ be as in Lemma \ref{lem2.4}, $k_{13}$ as in Lemma \ref{lem2.5} and $k_{16}$ as in Lemma \ref{lem2.6}. Applying this embedding to $v_{\eps x}$ and $u_{\eps x}$ and invoking \eqref{3.58}, we obtain
\bea{re2.62}
\|v_{\eps x}\|_{L^\infty(\Om)}\le c_1\sqrt{k_{16}},\qquad \|u_{\eps x}\|_{L^\infty(\Om)}\le c_1\sqrt{k_{16}}\qquad\text{for all }t\in(0,T)\cap(0,T_0)
\eea
and hence, by \eqref{3.29},
\bea{re2.63}
\|\Theta_{\eps x}\|_{L^\infty(\Om)}^2\le k_{13}\|v_{\eps x}\|_{L^\infty(\Om)}^2 +k_{13}\|u_{\eps x}\|_{L^\infty(\Om)}^2+k_{13}\le c_2:=k_{13}(2c_1^2k_{16}+1)
\eea
for all $t\in(0,T)\cap(0,T_0)$. Testing the fifth equation in \eqref{approx} against $-g_{\eps xx}$, integrating by parts and using \eqref{3.58}, \eqref{re2.62}, \eqref{re2.63}, as well as the boundary conditions $v_{\eps x}=u_{\eps x}=g_\eps=0$ on $\partial\Om$, we obtain, upon applying Young's inequality,
\bea{2.56}
\frac{\tau}{2}\frac{d}{dt}\io g_{\eps x}^2&=&-\eps\tau\io g_{\eps xx}^2+\io D_\eps'(\Theta_\eps)\Theta_{\eps x}v_{\eps x}g_{\eps x}+\io D_\eps(\Theta_\eps)v_{\eps xx}g_{\eps x}\nn\\
&&+\io C_\eps'(\Theta_\eps)\Theta_{\eps x}u_{\eps x}g_{\eps x}+\io C_\eps(\Theta_\eps)u_{\eps xx}g_{\eps x}-\io g_{\eps x}^2\nn\\
&\le & -\eps\tau\io g_{\eps xx}^2+k_7^2\io \Theta_{\eps x}^2v_{\eps x}^2+\frac{1}{4}\io g_{\eps x}^2+k_2^2\io v_{\eps xx}^2+\frac{1}{4}\io g_{\eps x}^2\nn\\
&&+k_8^2\io \Theta_{\eps x}^2u_{\eps x}^2+\frac{1}{4}\io g_{\eps x}^2+ k_4^2\io u_{\eps xx}^2+\frac{1}{4}\io g_{\eps x}^2-\io g_{\eps x}^2\nn\\
&\le& -\eps\tau\io g_{\eps xx}^2+c_3\qquad\text{for all }t\in(0,T)\cap(0,T_0)
\eea
with $c_3:=k_7^2k_{16}c_1^2c_2|\Om|+k_2^2k_{16}+k_8^2k_{16}c_1^2c_2|\Om|+k_4^2k_{16}$.
Analogously, testing the third equation in \eqref{approx}, we obtain
$$\frac{1}{2}\frac{d}{dt}\io u_{\eps xx}^2+\eps\io u_{\eps xxx}^2=\io u_{\eps xx}v_{\eps xx}\le\frac{1}{2}\io u_{\eps xx}^2+\frac{1}{2}\io v_{\eps xx}^2\le k_{16}\qquad\text{for all }t\in(0,T)\cap(0,T_0)$$
by Young's inequality and \eqref{3.58}. Integrating this and \eqref{2.56} over $(0,t)$, we obtain
\bea{nn}
\frac{\tau}{2}\io g_{\eps x}^2(\cdot,t)+\eps\tau\int_0^t\io g_{\eps xx}^2+\eps\int_0^t\io u_{\eps xxx}^2&\le& \frac{\tau}{2}\io g_{0\eps x}^2+\frac{1}{2}\io u_{0\eps xx}^2+(c_3+k_{16})t\nn\\
&\le& \frac{\tau}{2}\io g_{0\eps x}^2+\frac{1}{2}\io u_{0\eps xx}^2+(c_3+k_{16})T_0\nn
\eea
for all $t\in(0,T)\cap(0,T_0)$. Since $\eps\in(0,1)$, $(g_{0\eps})_{\eps\in(0,1)}$ is bounded in $W^{1,2}(\Om)$ and $(u_{0\eps})_{\eps\in(0,1)}$ is bounded in $W^{2,2}(\Om)$, we can find $k_{17}>0$ fulfilling \eqref{2.55.00}, \eqref{2.55.0} and \eqref{2.55}.
\qed
In preparation for the bootstrap argument establishing that \eqref{3.26} persists on an $\eps$- independent time interval, we collect in the following lemma the estimates on several auxiliary functions that will be needed there.
\begin{lem}\label{lem2.9}
Let $M$, $\eps_{\star\star}$ and $T_0$ be as in Lemma \ref{lem2.4} and Lemma \ref{lem2.6}. For $\eps\in(0,\eps_{\star\star})$, writing
$$a_\eps:=E_\eps(\Theta_\eps)v_{\eps x}+F_\eps(\Theta_\eps)h_\eps\qquad\text{in }\Om\times(0,\tme),$$
with $h_\eps$ as in \eqref{2.36.1}, and
\bea{2.57}
f_{1,\eps}:=a_\eps^2\qquad\text{in }\Om\times(0,\tme)
\eea
and
\bea{2.58}
f_{2,\eps}:=2a_\eps\Big\{E_\eps(\Theta_\eps)v_{\eps xt}+\frac{F_\eps(\Theta_\eps)}{\tau}\big(D_\eps(\Theta_\eps)v_{\eps xt}+C_\eps(\Theta_\eps)u_{\eps xt}-g_{\eps t}\big)\Big\}\quad\text{in }\Om\times(0,\tme)
\eea
as well as
\bea{2.59}
f_{3,\eps}:=2a_\eps\Big\{E_\eps'(\Theta_\eps)v_{\eps x}+F_\eps'(\Theta_\eps)h_\eps+\frac{F_\eps(\Theta_\eps)}{\tau}\big(D_\eps'(\Theta_\eps)v_{\eps x}+C_\eps'(\Theta_\eps)u_{\eps x}\big)\Big\}\quad\text{in }\Om\times(0,\tme),
\eea
so that $g_{\eps t}=\eps g_{\eps xx}+h_\eps$, $\Theta_{\eps t}=\Theta_{\eps xx}+f_{1,\eps}$ and $\Theta_{\eps tt}=(\Theta_{\eps t})_{xx}+f_{2,\eps}+f_{3,\eps}\Theta_{\eps t}$ in $\Om\times(0,\tme)$, there exists $k_{18}>0$ such that whenever $\eps\in(0,\eps_{\star\star})$ and $T\in(0,\tme)$ are such that \eqref{3.26} holds, 
\bea{2.60}
\|h_\eps\|_{L^\infty(\Om)}\le k_{18}\qquad\text{for all }t\in(0,T)\cap(0,T_0),
\eea
and
\bea{2.59.0}
\|a_\eps\|_{L^\infty(\Om)}\le k_{18}\qquad\text{for all }t\in(0,T)\cap(0,T_0),
\eea
and
\bea{2.61}
\|f_{1,\eps}\|_{L^1(\Om)}\le k_{18}\qquad\text{for all }t\in(0,T)\cap(0,T_0),
\eea
and
\bea{2.62}
\int_0^t\io f_{2,\eps}^2\le k_{18}\qquad\text{for all }t\in(0,T)\cap(0,T_0),
\eea
as well as
\bea{2.63}
\|f_{3,\eps}\|_{L^\infty(\Om)}\le k_{18}\qquad\text{for all }t\in(0,T)\cap(0,T_0).
\eea
\end{lem}
\proof
By the continuity of the embedding $W^{1,2}(\Om)\hookrightarrow L^\infty(\Om)$ we can pick some $c_1>0$ fulfilling
$$\|\varphi\|_{L^\infty(\Om)}\le c_1\|\varphi_x\|_{L^2(\Om)}\qquad\text{for all }\varphi\in W_0^{1,2}(\Om),$$
and let $(k_i)_{i\in\{1,\dots,17\}}$ as in Lemma \ref{lem2.4}-\ref{lem2.8}. Then, by \eqref{3.26}, \eqref{3.27} and \eqref{3.58},
\bea{re2.73}
\|h_\eps\|_{L^\infty(\Om)}&\le&\frac{k_2}{\tau}\|v_{\eps x}\|_{L^\infty(\Om)}+\frac{k_4}{\tau}\|u_{\eps x}\|_{L^\infty(\Om)}+\frac{1}{\tau}\|g_\eps\|_{L^\infty(\Om)}\nn\\
&\le& c_2:=\frac{c_1\sqrt{k_{16}}}{\tau}\cdot(k_2+k_4)+\frac{M}{\tau}\qquad\text{for all }t\in(0,T)\cap(0,T_0).
\eea
Consequently, by \eqref{3.27} and \eqref{3.58},
\bea{re2.74}
\|a_\eps\|_{L^\infty(\Om)}\le k_5\|v_{\eps x}\|_{L^\infty(\Om)}+c_2k_6\le c_3:=c_1k_5\sqrt{k_{16}}+c_2k_6\qquad\text{for all }t\in(0,T)\cap(0,T_0),
\eea
so that
\bea{nn}
\|f_{1,\eps}\|_{L^1(\Om)}=\io a_\eps^2\le c_4:=c_3^2|\Om|\qquad\text{for all }t\in(0,T)\cap(0,T_0).
\eea
Moreover, by \eqref{3.27}-\eqref{3.28.1}, \eqref{3.58}, \eqref{re2.73} and \eqref{re2.74}
\bea{nn}
\|f_{3,\eps}\|_{L^\infty(\Om)}&\le& 2c_3\Big\{k_{11}\|v_{\eps x}\|_{L^\infty(\Om)}+c_2k_{12}+\frac{k_6k_7}{\tau}\|v_{\eps x}\|_{L^\infty(\Om)}+\frac{k_6k_8}{\tau}\|u_{\eps x}\|_{L^\infty(\Om)}\Big\}\nn\\
&\le&c_5:=2c_3\Big\{c_1k_{11}\sqrt{k_{16}}+c_2k_{12}+\frac{c_1k_6k_7\sqrt{k_{16}}}{\tau}+\frac{c_1k_6k_8\sqrt{k_{16}}}{\tau}\Big\}
\eea
for all $t\in(0,T)\cap(0,T_0)$. Finally, with $c_6:=16c_3^2k_5^2+\frac{16c_3^2k_6^2}{\tau^2}\cdot\max\big\{k_2^2,k_4^2,1\big\}$, we obtain from \eqref{3.27}, \eqref{3.58}, \eqref{2.55} and \eqref{re2.74} that
\bea{nn}
\int_0^t\io f_{2,\eps}^2(\cdot,s)ds&\le&16c_3^2k_5^2\int_0^t \|v_{\eps xt}(\cdot,s)\|_{L^2(\Om)}^2ds\nn\\
&&+\frac{16c_3^2k_6^2}{\tau^2}\int_0^t\Big\{k_2^2\|v_{\eps xt}(\cdot,s)\|_{L^2(\Om)}^2+k_4^2\|u_{\eps xt}(\cdot,s)\|_{L^2(\Om)}^2+\|g_{\eps t}(\cdot,s)\|_{L^2(\Om)}^2\Big\}ds\nn\\
&\le& c_6\int_0^t\|v_{\eps xt}(\cdot,s)\|_{L^2(\Om)}^2ds+c_6\int_0^t\|u_{\eps xt}(\cdot,s)\|_{L^2(\Om)}^2ds+c_6\int_0^t\|g_{\eps t}(\cdot,s)\|_{L^2(\Om)}^2ds\nn\\
&=&c_6\int_0^t\|\eps v_{\eps xxx}(\cdot,s)+w_{\eps x}(\cdot,s)\|_{L^2(\Om)}^2ds+c_6\int_0^t\|\eps u_{\eps xxx}(\cdot,s)+v_{\eps x}(\cdot,s)\|_{L^2(\Om)}^2ds\nn\\
&&+c_6\int_0^t\|\eps g_{\eps xx}(\cdot,s)+h_\eps(\cdot,s)\|_{L^2(\Om)}^2ds\nn\\
&\le& 2c_6\int_0^t \eps^2\|v_{\eps xxx}(\cdot,s)\|_{L^2(\Om)}^2+\|w_{\eps x}(\cdot,s)\|_{L^2(\Om)}^2ds\nn\\
&&+2c_6\int_0^t \eps^2\|u_{\eps xxx}(\cdot,s)\|_{L^2(\Om)}^2+\|v_{\eps x}(\cdot,s)\|_{L^2(\Om)}^2ds\nn\\
&&+2c_6\int_0^t \eps^2\|g_{\eps xx}(\cdot,s)\|_{L^2(\Om)}^2+\|h_{\eps}(\cdot,s)\|_{L^2(\Om)}^2ds\nn\\
&\le& c_7:=2c_6\cdot\{\eps k_{16}+k_{16}T_0+2\eps k_{17}+c_1^2k_{16}|\Om|T_0+c_2^2|\Om|T_0\}
\eea
for all $t\in(0,T)\cap(0,T_0)$, because $\eps\le 1$. So \eqref{2.60}-\eqref{2.63} follow with $k_{18}:=\max\{c_2,c_4,c_5,c_7\}$.
\qed
Equipped with these bounds, we now turn to showing that \eqref{3.26} itself holds on a time interval independent of $\eps$.
\begin{lem}\label{lem2.7}
Let $M$, $\eps_{\star\star}$ and $T_0$ be as provided by Lemma \ref{lem2.4} and Lemma \ref{lem2.6}. Then one can find $T_\star\in(0,T_0)$ such that
\bea{3.67}
\|\Theta_\eps(\cdot,t)\|_{L^\infty(\Om)}+\|\Theta_{\eps t}(\cdot,t)\|_{L^\infty(\Om)}+\|g_\eps(\cdot,t)\|_{L^\infty(\Om)}\le M,
\eea
for all $t\in(0,T_\star)\cap(0,\tme)$ and $\eps\in(0,\eps_{\star\star})$.
\end{lem}
\proof
We recall once more the smoothing properties of the Neumann heat semigroup $(e^{t\Del})_{t\ge 0}$ on $\Om$ (\cite{Winkler10}): there exist $c_1>0$ and $c_2>0$ such that whenever $\varphi\in C^0(\ov\Om)$,
\bea{3.68}
\|e^{t\Del}\varphi\|_{L^\infty(\Om)}\le c_1t^{-\frac{1}{2}}\|\varphi\|_{L^1(\Om)}\qquad\text{for all }t\in(0,T_0)
\eea
and
\bea{3.69}
\|e^{t\Del}\varphi\|_{L^\infty(\Om)}\le c_2 t^{-\frac{1}{4}}\|\varphi\|_{L^2(\Om)}\qquad\text{for all }t\in(0,T_0).
\eea
Taking $h_\eps$, $f_{1,\eps}$, $f_{2,\eps}$, $f_{3,\eps}$ and $k_{18}$ as in Lemma \ref{lem2.9}, we then let $T_\star\in(0,T_0)$ be small enough such that
\bea{3.71}
2c_1k_{18}T_\star^\frac{1}{2}\le \frac{M}{12},\qquad \sqrt{2k_{18}}c_2T_\star^\frac{1}{4}+k_{18}MT_\star\le\frac{M}{12}\qquad\text{and}\qquad k_{18}T_\star\le\frac{M}{12}.
\eea
Due to \eqref{3.25} and the continuity of $\Theta_\eps$, $\Theta_{\eps t}$ and $g_\eps$ in $\ov\Om\times[0,\tme)$, for fixed $\eps\in(0,\eps_{\star\star})$,
\bea{3.73}
T_\eps:=\sup\Big\{T'\in(0,\tme)\ \Big|\ \|\Theta_\eps(\cdot,t)\|_{L^\infty(\Om)}+\|\Theta_{\eps t}(\cdot,t)\|_{L^\infty(\Om)}+\|g_\eps(\cdot,t)\|_{L^\infty(\Om)}\nn\\
\le M\text{ for all }t\in(0,T')\Big\}\hspace{70mm}
\eea
is a well-defined element of $(0,\tme]\subset(0,\infty]$, and actually satisfies
$$T_\eps\ge T_\star.$$
To see this, we observe that according to a Duhamel representation associated with the fourth sub-problem in \eqref{approx}, since $e^{t\Del}$ is nonexpansive on $L^\infty(\Om)$ according to the maximum principle, Lemma \ref{lem2.4}, \eqref{3.68} and \eqref{3.71} imply that
\bea{3.81}
\|\Theta_\eps(\cdot,t)\|_{L^\infty(\Om)}&=&\bigg\|e^{t\Del}\Theta_{0\eps}+\int_0^t e^{(t-s)\Del}f_{1,\eps}(\cdot,s)ds\bigg\|_{L^\infty(\Om)}\nn\\
&\le &\|\Theta_{0\eps}\|_{L^\infty(\Om)}+c_1\int_0^t (t-s)^{-\frac{1}{2}}\|f_{1,\eps}(\cdot,s)\|_{L^1(\Om)} ds\nn\\
&\le& \frac{M}{12}+2c_1k_{18}t^{\frac{1}{2}}\nn\\
&\le&\frac{M}{6}\qquad\text{for all }t\in(0,T_\eps).
\eea
Similarly, using that $\Theta_{\eps tt}=(\Theta_{\eps t})_{xx}+f_{2,\eps}+f_{3,\eps}\Theta_{\eps t}$ in $\Om\times(0,\tme)$ and $(\Theta_{\eps t})_x=0$ on $\partial\Om\times(0,\tme)$ by \eqref{approx}, \eqref{2.58} and \eqref{2.59}, on employing Cauchy-Schwarz inequality, Lemma \ref{lem2.4}, \eqref{3.69} and \eqref{3.71} we see that
\bea{2.77}
\|\Theta_{\eps t}(\cdot,t)\|_{L^\infty(\Om)}&=&\bigg\|e^{t\Del}\Theta_{0\eps t}+\int_0^t e^{(t-s)\Del}f_{2,\eps}(\cdot,s)ds+\int_0^t e^{(t-s)\Del}\big\{f_{3,\eps}(\cdot,s)\Theta_{\eps t}(\cdot,s)\big\}ds\bigg\|_{L^\infty(\Om)}\nn\\
&\le&\|\Theta_{0\eps t}\|_{L^\infty(\Om)}+c_2\int_0^t(t-s)^{-\frac{1}{4}}\|f_{2,\eps}(\cdot,s)\|_{L^2(\Om)}ds\nn\\
&&+\int_0^t\|f_{3,\eps}(\cdot,s)\|_{L^\infty(\Om)}\|\Theta_{\eps t}(\cdot,s)\|_{L^\infty(\Om)}ds\nn\\
&\le&\|\Theta_{0\eps t}\|_{L^\infty(\Om)}+c_2\cdot\bigg\{\int_0^t(t-s)^{-\frac{1}{2}}ds\bigg\}^\frac{1}{2}\cdot\bigg\{\int_0^t\|f_{2,\eps}(\cdot,s)\|_{L^2(\Om)}^2ds\bigg\}^\frac{1}{2}\nn\\
&&+\int_0^t\|f_{3,\eps}(\cdot,s)\|_{L^\infty(\Om)}\|\Theta_{\eps t}(\cdot,s)\|_{L^\infty(\Om)}ds\nn\\
&\le&\frac{M}{12}+\sqrt{2k_{18}}c_2t^\frac{1}{4}+ k_{18}Mt\nn\\
&\le&\frac{M}{6}\qquad\text{for all }t\in(0,T_\eps).
\eea
For $g_\eps$, since $e^{\eps t\Del_D}$, the semigroup generated by the Dirichlet realization $\Del_D$ of the Laplacian, is likewise nonexpansive on $L^\infty(\Om)$ according to the maximum principle and \eqref{3.71} imply that
\bea{2.78}
\|g_\eps(\cdot,t)\|_{L^\infty(\Om)}&=&\bigg\|e^{\eps t\Del_D}g_{0\eps}+\int_0^t e^{\eps(t-s)\Del_D}h_\eps(\cdot,s)ds\bigg\|_{L^\infty(\Om)}\nn\\
&\le& \|g_{0\eps}\|_{L^\infty(\Om)}+\int_0^t\|h_\eps(\cdot,s)\|_{L^\infty(\Om)}ds\nn\\
&\le&\frac{M}{12}+k_{18}t\nn\\
&\le&\frac{M}{6}\qquad\text{for all }t\in(0,T_\eps).
\eea
From \eqref{3.81}, \eqref{2.77} and \eqref{2.78} we thus infer that
$$\|\Theta_\eps(\cdot,t)\|_{L^\infty(\Om)}+\|\Theta_{\eps t}(\cdot,t)\|_{L^\infty(\Om)}+\|g_\eps(\cdot,t)\|_{L^\infty(\Om)}\le\frac{M}{2}\qquad\text{for all }t\in(0,T_\eps),$$
which by continuity of $\Theta_\eps$, $\Theta_{\eps t}$ and $g_\eps$ is compatible with \eqref{3.73} only if indeed $T_\eps\ge T_\star$.
\qed
Consequently, our approximate solutions persists at least up to the time obtained above, satisfying estimates that follow directly from Lemma \ref{lem2.6} and Lemma \ref{lem2.7}.
\begin{cor}\label{cor2.10}
Let $\eps_{\star\star}$ and $T_\star$ be as in Lemma \ref{lem2.6} and Lemma \ref{lem2.7}. Then $\tme\ge T_\star$ for all $\eps\in(0,\eps_{\star\star})$, and there exists $K>0$ such that
\bea{3.83}
\|w_\eps(\cdot,t)\|_{W^{1,2}(\Om)}\le K\qquad\text{for all }t\in(0,T_\star)\text{ and }\eps\in(0,\eps_{\star\star}),
\eea
that
\bea{3.84}
\|v_\eps(\cdot,t)\|_{W^{2,2}(\Om)}\le K\qquad\text{for all }t\in(0,T_\star)\text{ and }\eps\in(0,\eps_{\star\star}),
\eea
that
\bea{3.85}
\|u_\eps(\cdot,t)\|_{W^{2,2}(\Om)}\le K\qquad\text{for all }t\in(0,T_\star)\text{ and }\eps\in(0,\eps_{\star\star}),
\eea
that
\bea{3.86}
\|\Theta_{\eps t}(\cdot,t)\|_{L^\infty(\Om)}\le K\qquad\text{for all }t\in(0,T_\star)\text{ and }\eps\in(0,\eps_{\star\star}),
\eea
that
\bea{3.87}
\|\Theta_\eps(\cdot,t)\|_{W^{2,\infty}(\Om)}\le K\qquad\text{for all }t\in(0,T_\star)\text{ and }\eps\in(0,\eps_{\star\star}),
\eea
and that
\bea{2.83}
\|g_{\eps}(\cdot,t)\|_{W^{1,2}(\Om)}\le K\qquad\text{for all }t\in(0,T_\star)\text{ and }\eps\in(0,\eps_{\star\star}).
\eea
\end{cor}
\proof
Thanks to the continuity of the embedding $W^{1,2}(\Om)\hookrightarrow L^\infty(\Om)$, all claimed properties immediately result from Lemma \ref{lem2.3}, Lemma \ref{lem2.5}, Lemma \ref{lem2.6}, Lemma \ref{lem2.8} and Lemma \ref{lem2.7}.
\qed
Beyond the spatial bounds collected so far, we shall also need some control on the time derivatives of the approximate solutions.
\begin{lem}\label{lem2.11}
There exists $K>0$ such that with $\eps_{\star\star}$ and $T_\star$ as in Lemma \ref{lem2.6} and Lemma \ref{lem2.7},
\bea{3.88}
\|w_{\eps t}(\cdot,t)\|_{(W^{1,2}(\Om))^\star}\le K,\qquad\|v_{\eps t}(\cdot,t)\|_{L^2(\Om)}\le K\qquad\text{and}\qquad \|u_{\eps t}(\cdot,t)\|_{L^2(\Om)}\le K
\eea
for all $t\in(0,T_\star)$ and $\eps\in(0,\eps_{\star\star})$.
\end{lem}
\proof
Since the first, second and third equations in \eqref{approx} coincide, under the substitution described in the Introduction, with those in \cite{CW26}, this is a direct consequence of \cite[Lemma 3.13]{CW26}.
\qed
It remains to obtain an analogous bound for the time derivate of $g_\eps$, which follows by a direct computation.
\begin{lem}\label{lem2.12}
With $\eps_{\star\star}$ and $T_\star$ as in Lemma \ref{lem2.6} and Lemma \ref{lem2.7}, there exists $K>0$ such that
\bea{3.90}
\int_0^{T_\star}\io g_{\eps t}^2\le K\qquad\text{for all }\eps\in(0,\eps_{\star\star}).
\eea
\end{lem}
\proof
By the fifth equation in \eqref{approx} and with $h_\eps$ as in \eqref{2.36.1}, Young's inequality yields
$$\int_0^{T_\star}\io g_{\eps t}^2\le 2\int_0^{T_\star}\io \eps^2 g_{\eps xx}^2+2\int_0^{T_\star}\io h_\eps^2\le 2\eps\int_0^{T_\star}\io g_{\eps xx}^2+2\|h_\eps\|_{L^\infty(\Om)}^2|\Om|T_\star\qquad\text{for all }\eps\in(0,1),$$
where the right-hand side is bounded, uniformly in $\eps\in(0,\eps_{\star\star})$, by Lemma \ref{lem2.8} and by \eqref{2.60}.
\qed
\subsection{Passage to the limit and proof of Theorem \ref{theo1}}\label{sec2.3}
The prior estimates established in Section \ref{sec2.2} now put us in a position to construct a strong solution of \eqref{0} as a limit of the approximate solutions of \eqref{approx}, defined up to the time $T_\star$ from Lemma \ref{lem2.7}:
\begin{lem}\label{lem2.13}
If $\eps_{\star\star}$ and $T_\star$ are as found in Lemma \ref{lem2.6} and Lemma \ref{lem2.7}, then one can find $(e_j)_{j\in\mathbb{N}}\subset(0,\eps_{\star\star})$ as well as functions
\be{3.91}
	\lball
	u\in C^0([0,T_\star];W_N^{2,2}(\Om)),\\
	\Theta\in C^0([0,T_\star];C^1(\ov\Om))\cap C^{2,1}(\ov\Om\times(0,T_\star))\cap W^{1,\infty}(\Om\times(0,T_\star))\qquad \text{and}\\
	g\in C^1([0,T_\star];C^0(\ov\Om))
	\ear
\ee
with
\bea{3.92}
u_t\in C^0([0,T_\star];C^1(\ov\Om))\cap L^\infty((0,T_\star);W_N^{2,2}(\Om))
\eea
and
\bea{3.93}
u_{tt}\in C^0(\ov\Om\times[0,T_\star])\cap L^\infty((0,T_\star);W^{1,2}(\Om))
\eea
such that $\Theta\ge 0$ in $\Om\times(0,T_\star)$, that $e_j\searrow 0$ as $j\to\infty$, and that
\bea{nn}
&&u_\eps\to u\qquad\text{in }C^0([0,T_\star];C^1(\ov\Om)),\label{3.94}\\
&&v_\eps\to u_t\qquad\text{in }C^0([0,T_\star];C^1(\ov\Om)),\label{3.95}\\
&&w_\eps\to u_{tt}\qquad\text{in }C^0(\ov\Om\times[0,T_\star]),\label{3.96}\\
&&\Theta_\eps\to\Theta\qquad\text{in }C^0([0,T_\star];C^1(\ov\Om)),\label{3.97}\\
&&\Theta_{\eps t}\stackrel{\ast}{\rightharpoonup}\Theta_t\qquad\text{in }L^\infty(\Om\times(0,T_\star)),\label{3.98}\\
&&\Theta_{\eps xx}\stackrel{\ast}{\rightharpoonup}\Theta_{xx}\qquad\text{in }L^\infty(\Om\times(0,T_\star)),\qquad\label{3.99}\qquad\text{and}\\
&&g_\eps\to g\qquad\text{in }C^0([0,T_\star];C^0(\ov\Om))\label{2.96}
\eea
as $\eps=\eps_j\searrow 0$, and that $(u,\Theta,g)$ forms a strong solution of \eqref{0} in $\Om\times(0,T_\star)$ in the sense specified in Definition \ref{def2.1}.
\end{lem}
\proof
As in \cite[proof of Lemma 3.14]{CW26}, we combine the bounds from Corollary \ref{cor2.10}, Lemma \ref{lem2.11} and Lemma \ref{lem2.12} with five (rather than four) applications of an Aubin-Lions lemma (\cite{Aubin}), the fifth based on the compact embedding $W^{1,2}(\Om)\hookrightarrow C^0(\ov\Om)$ and controlling $g_\eps$. This provides $(e_j)_{j\in\mathbb{N}}\subset(0,\eps_{\star\star})$ such that $e_j\searrow 0$ as $j\to\infty$ and such that \eqref{3.94}-\eqref{2.96} hold, that 
\bea{2.99.1}
\Theta\in C^0([0,T_\star];C^1(\ov\Om))\cap L^\infty((0,T_\star);W^{2,\infty}(\Om))\qquad\text{and}\qquad g\in C^0([0,T_\star];C^0(\ov\Om)),
\eea
and that $u_t=v$ and $v_t=w$, so that $u$ already satisfies \eqref{3.91}-\eqref{3.93}. The nonnegativity of $\Theta$ is trivially inherited from that of $\Theta_\eps$ for $\eps\in(0,\eps_{\star\star})$ through \eqref{3.97}, and \eqref{2.4} as well as the properties $u(\cdot,0)=u_0$, $u_t(\cdot,0)=u_{0t}$, $\Theta(\cdot,0)=\Theta_0$ and $g(\cdot,0)=g_0$ can be derived using \eqref{3.94}-\eqref{2.96} together with \eqref{2.35} in a straightforward manner. Since \eqref{3.94}, \eqref{3.95} and \eqref{3.97} ensure that for each $\varphi\in C_0^\infty(\ov\Om\times[0,T_\star))$, in the identity
$$-\tau\int_0^{T_\star}\io g_\eps\varphi_t-\tau\io g_{0\eps}\varphi(\cdot,0)+\int_0^{T_\star}\io g_\eps\varphi=\int_0^{T_\star}\io \big\{D_\eps(\Theta_\eps)v_{\eps x}+C_\eps(\Theta_\eps)u_{\eps x}\big\}\cdot\varphi-\eps\tau\int_0^{T_\star}\io g_{\eps x}\varphi_x,$$
for $\eps\in(0,1)$, the last term on the right tends to $0$ as $\eps\searrow 0$, by Cauchy-Schwarz inequality and since $g_{\eps}$ is bounded in $W^{1,2}(\Om)$ uniformly in $t$ by Corollary \ref{cor2.10}, we may let $\eps=\eps_j\searrow 0$ to verify that
$$-\tau\int_0^{T_\star}\io g\varphi_t-\tau\io g_{0}\varphi(\cdot,0)+\int_0^{T_\star}\io g\varphi=\int_0^{T_\star}\io F\varphi,\qquad F:=D(\Theta)u_{xt}+C(\Theta)u_x,$$
we furthermore obtain that the problem
\be{2.100}
	\lball
	\tau g_t+g=F(x,t),& x\in\Om,\ t\in(0,T_\star),\\[1mm]
	g=0,& x\in\partial\Om,\ t\in(0,T_\star),\\[1mm]
	g(x,0)=g_0(x),& x\in\Om,
	\ear
\ee
is solved in the standard weak sense. Because $F\in C^0(\ov\Om\times[0,T_\star])$ by the regularity of $\Theta$, $u$ and $u_t$ and because $g\in C^0([0,T_\star];C^0(\ov\Om))$, for each fixed $x\in\Om$, variation of constants applied to \eqref{2.100} shows that
\bea{2.102}
g(x,t)=g_0(x)e^{-\frac{t}{\tau}}+\frac{1}{\tau}\int_0^t e^{-\frac{t-s}{\tau}}F(x,s)ds,\qquad\text{for all }t\in[0,T_\star].
\eea
In particular, $g$ is classically differentiable with respect to $t$, with $\tau g_t=F-g$ in $\ov\Om\times[0,T_\star]$, since the right-hand side belongs to $C^0(\ov\Om\times[0,T_\star])$, we obtain $g\in C^1([0,T_\star];C^0(\ov\Om))$, as claimed in \eqref{3.91}, and that $g$ solves the corresponding equation in the classical sense. We now single out two further properties of $g$, understood in the weak (distributional) sense. Since $u$ satisfies \eqref{3.91}-\eqref{3.93}, and $\Theta$ satisfies \eqref{2.99.1} with $\Theta_t\in L^\infty(\Om\times(0,T_\star))$ by \eqref{3.98}, differentiating $F$, justified by the standard chain rule for the composition of $C^1$ functions $D$, $C$ with $\Theta\in W^{1,\infty}$, together with the (weak) Leibniz rule for the product of the resulting $L^\infty$ factors $D(\Theta)$, $C(\Theta)$ with the $W^{1,2}$ functions $u_{xt}$, $u_x$ (see, e.g. \cite{Friedman69}), shows that
\bea{nn}
F_x=D'(\Theta)\Theta_x u_{xt}+D(\Theta)u_{xxt}+C'(\Theta)\Theta_x u_x+C(\Theta)u_{xx}\nn
\eea
and
\bea{nn}
F_t=D'(\Theta)\Theta_t u_{xt}+D(\Theta)u_{xtt}+C'(\Theta)\Theta_t u_x+C(\Theta)u_{xt}\nn
\eea
both belong to $L^\infty((0,T_\star);L^2(\Om))$. Since moreover $g_0\in W^{1,2}(\Om)$, and since \eqref{2.102} holds for every $x\in\Om$, we may differentiate it with respect to $x$, showing that $g$ is differentiable with respect to $x$, with
$$g_x(\cdot,t)=g_{0x}e^{-\frac{t}{\tau}}+\frac{1}{\tau}\int_0^t e^{-\frac{t-s}{\tau}}F_x(\cdot,s)ds,\qquad\text{for all }t\in[0,T_\star],$$ so that $g_x\in L^\infty((0,T_\star);L^2(\Om))$. Differentiating $\tau g_t=F-g$ in $\ov\Om\times[0,T_\star]$ with respect to $x$ and $t$ then shows that also
\bea{2.102.1}
g_{xt},g_{tt}\in L^\infty((0,T_\star);L^2(\Om)).
\eea
It remains to upgrade the regularity of $\Theta$ from \eqref{2.99.1} to the classical regularity claimed in \eqref{3.91}. To this end, we once more draw on \eqref{3.94}, \eqref{3.95}, \eqref{3.97} and \eqref{2.96}, which, applied to $h_\eps$ as in \eqref{2.36.1}, in particular yield
$$h_\eps=\frac{1}{\tau}\big\{D_\eps(\Theta_\eps)v_{\eps x}+C_\eps(\Theta_\eps)u_{\eps x}-g_\eps\big\}\to\frac{1}{\tau}\big\{F-g\big\}=g_t$$ as $\eps=\eps_j\searrow 0$. For each $\psi\in C_0^\infty(\ov\Om\times[0,T_\star))$, we may pass to the limit in the identity
$$-\int_0^{T_\star}\io\Theta_\eps\psi_t-\io\Theta_{0\eps}\psi(\cdot,0)=-\int_0^{T_\star}\io\Theta_{\eps x}\psi_x+\int_0^{T_\star}\io|E_\eps(\Theta_\eps)v_{\eps x}+F_\eps(\Theta_\eps)h_\eps|^2\psi,$$ for $\eps\in(0,1)$, to verify that
$$-\int_0^{T_\star}\io\Theta\psi_t-\io\Theta_{0}\psi(\cdot,0)=-\int_0^{T_\star}\io\Theta_{x}\psi_x+\int_0^{T_\star}\io H\psi,\qquad H:=|E(\Theta)u_{xt}+F(\Theta)g_t|^2,$$ and we furthermore obtain that the problem
\be{2.103}
	\lball
	\Theta_t=\Theta_{xx}+H(x,t),& x\in\Om,\ t\in(0,T_\star),\\[1mm]
	\Theta_x=0,& x\in\partial\Om,\ t\in(0,T_\star),\\[1mm]
	\Theta(x,0)=\Theta_0(x),& x\in\Om,
	\ear
\ee
is solved in the standard weak sense specified, e.g., in \cite{LSU}. We claim that
\bea{2.104}
H\in C^0([0,T_\star];C^0(\ov\Om))\cap L^\infty((0,T_\star);W^{1,2}(\Om))\qquad\text{and}\qquad H_t\in L^\infty((0,T_\star);L^2(\Om)).
\eea
As for $F$ above, this follows from \eqref{3.91}-\eqref{3.93}, \eqref{3.98} and \eqref{2.102.1}. By a Gagliardo-Nirenberg interpolation as in \cite[proof of Lemma 3.14]{CW26}, \eqref{2.104} implies that $H\in C^\frac{1}{2}([0,T_\star];C^0(\ov\Om))$; together with the Hölder continuity in $x$ following from $H\in L^\infty((0,T_\star);W^{1,2}(\Om))$ via the compact embedding $W^{1,2}(\Om)\hookrightarrow C^{\vartheta_1}(\ov\Om)$ for some $\vartheta_1\in(0,1)$, this shows that $H\in C^{\vartheta,\frac{\vartheta}{2}}(\ov\Om\times[0,T_\star])$ for some $\vartheta\in(0,1)$. Standard parabolic theory (\cite{LSU}, \cite{PV}) then yields $\Theta\in C^{2,1}(\ov\Om\times(0,T_\star))$, as claimed in \eqref{3.91} and forms a classical solution of the respective sub-solution in \eqref{0}.
\qed
With Lemma \ref{lem2.13} at hand, we can now finally establish Theorem \ref{theo1} by means of a straightforward extension argument as well as a simple generalization to arbitrary initial data not necessarily satisfying \eqref{---}.\abs\abs
{\sc Proof of Theorem \ref{theo1}.\quad} We first consider the case in which \eqref{---} holds, that is, when the numbers
\bea{4.1}
y_0:=\frac{1}{|\Om|}\io u_0,\qquad y_1:=\frac{1}{|\Om|}\io u_{0t}\qquad\text{and}\qquad y_2:=\frac{1}{|\Om|}\io u_{0tt}
\eea
satisfy $y_0=y_1=y_2=0$. As in \cite[proof of Theorem \ref{theo1}]{CW26}, we may draw on Lemma \ref{lem2.13} to see that
$$\tm:=\sup\Big\{T>0\ \Big|\ \text{There exists a strong solution of \eqref{0} in }\Om\times(0,T)\Big\}$$
is a well-defined element of $(0,\infty]$ which, by Lemma \ref{lem2.2}, has the property that there exists exactly one strong solution $(u,\Theta,g)$ of \eqref{0} in all of $\Om\times(0,\tm)$. Now if this solution violated \eqref{1.16}, then $\tm<\infty$ and
\bea{4.2}
\|u_t\|_{W^{2,2}(\Om)}+\|u_{tt}\|_{W^{1,2}(\Om)}+\|\Theta\|_{W^{2,\infty}(\Om)}\le c_1\qquad\text{for all }t\in(t_0,\tm)
\eea
with some $c_1>0$ and some $t_0\in(0,\tm)$, which particularly would imply that
\bea{4.3}
\|u(\cdot,t)\|_{W^{2,2}(\Om)}=\bigg\|u(\cdot,t_0)+\int_{t_0}^t u_t(\cdot,s)ds\bigg\|_{W^{2,2}(\Om)}\le \|u(\cdot,t_0)\|_{W^{2,2}(\Om)}+(\tm-t_0)\cdot c_1
\eea
for all $t\in(t_0,\tm)$, and thus also
\bea{2.107}
\big\|D(\Theta)u_{xt}+C(\Theta)u_x\big\|_{W^{1,2}(\Om)}\le c_2\qquad\text{for all }t\in(t_0,\tm)\nn
\eea
for some $c_2=c_2(c_1)>0$. This shows that
\bea{2.108}
\|g(\cdot,t)\|_{W^{1,2}(\Om)}&=&\bigg\|g(\cdot,t_0)e^{-\frac{t-t_0}{\tau}}+\frac{1}{\tau}\int_{t_0}^t e^{-\frac{t-s}{\tau}}\big\{D(\Theta(\cdot,s))u_{xt}(\cdot,s)+C(\Theta(\cdot,s))u_x(\cdot,s)\big\}\bigg\|_{W^{1,2}(\Om)}ds\nn\\
&\le& \|g(\cdot,t_0)\|_{W^{1,2}(\Om)}+\frac{1}{\tau}\int_{t_0}^{\tm}\big\|D(\Theta(\cdot,s))u_{xt}(\cdot,s)+C(\Theta(\cdot,s))u_x(\cdot,s)\big\|_{W^{1,2}(\Om)}ds\nn\\
&\le& \|g(\cdot,t_0)\|_{W^{1,2}(\Om)}+\frac{\tm-t_0}{\tau}\cdot c_2\qquad\text{for all }t\in(t_0,\tm).
\eea
On the basis of \eqref{4.2}-\eqref{2.108}, we could accordingly fix $(t_j)_{j\in\mathbb{N}}\subset(t_0,\tm)$ as well as $\tilde{u}_0,\tilde{u}_{0t}\in W_N^{2,2}(\Om)$, $\tilde{u}_{0tt},\tilde{g}_0\in W^{1,2}(\Om)$ and $\tilde{\Theta}_0\in W_N^{2,\infty}(\Om;[0,\infty))$ such that $t_j\nearrow\tm$ as $j\to\infty$, that
$$u(\cdot,t_j)\rightharpoonup \tilde{u}_0\qquad\text{and}\qquad u_t(\cdot,t_j)\rightharpoonup\tilde{u}_{0t}\qquad\text{in }W^{2,2}(\Om)\qquad\text{as }j\to\infty,$$
that
$$u_{tt}(\cdot,t_j)\rightharpoonup\tilde{u}_{0tt}\qquad\text{and}\qquad g(\cdot,t_j)\rightharpoonup\tilde{g}_0\qquad\text{in }W^{1,2}(\Om)\qquad\text{as }j\to\infty,$$
and that
$$\Theta(\cdot,t_j)\stackrel{\ast}{\rightharpoonup}\tilde{\Theta}_0\qquad\text{in }W^{2,\infty}(\Om)\qquad\text{as }j\to\infty.$$
Applying Lemma \ref{lem2.13} once more, now with initial data $(\tilde{u}_0,\tilde{u}_{0t},\tilde{u}_{0tt},\tilde{\Theta}_0,\tilde{g}_0)$ and with $t$ shifted by $\tm$, we obtain $T>\tm$ together with a strong solution $(\tilde{u},\tilde{\Theta},\tilde{g})$ of \eqref{0} in $\Om\times(\tm,T)$ satisfying $(\tilde{u},\tilde{u}_t,\tilde{u}_{tt},\tilde{\Theta},\tilde{g})|_{t=\tm}=(\tilde{u}_0,\tilde{u}_{0t},\tilde{u}_{0tt},\tilde{\Theta}_0,\tilde{g}_0)$. Gluing $(u,\Theta,g)$ and $(\tilde{u},\tilde{\Theta},\tilde{g})$ together across $t=\tm$ then produces a strong solution of \eqref{0} on all $\Om\times(0,T)$, contradicting the maximality of $\tm$.\abs 
Having thus completed the proof in the case when \eqref{---} holds, we are left with the general situation in which the quantities $y_0$, $y_1$ and $y_2$ are arbitrary real numbers. As in \cite[proof of Theorem 1.1]{CW26}, the functions $\tilde{u}_0:=u_0-y_0$, $\tilde{u}_{0t}:=u_{0t}-y_1$ and $\tilde{u}_{0tt}:=u_{0tt}-y_2$ satisfy $\io\tilde{u}_0=\io\tilde{u}_{0t}=\io\tilde{u}_{0tt}=0$, so that, repeating the above with $(u_0,u_{0t},u_{0tt})$ replaced by $(\tilde{u}_0,\tilde{u}_{0t},\tilde{u}_{0tt})$, we obtain a strong solution $(\tilde{u},\tilde{\Theta},\tilde{g})$ of \eqref{0} in $\Om\times(0,\tm)$ satisfying \eqref{1.16}. Letting $y\in C^\infty([0,\infty))$ solve $\tau y'''+y''=0$ in $(0,\infty)$ with $y(0)=y_0$, $y'(0)=y_1$ and $y''(0)=y_2$, and setting $u(x,t):=\tilde{u}(x,t)+y(t)$, $\Theta(x,t):=\tilde{\Theta}(x,t)$ and $g(x,t):=\tilde{g}(x,t)$ for $(x,t)\in\ov\Om\times[0,\tm)$, we obtain the desired strong solution of \eqref{0} on $\Om\times(0,\tm)$ which satisfies \eqref{1.13}-\eqref{1.16}.
\qed
\mysection{Classical solutions}\label{classical}
Having established local strong solvability in Theorem \ref{theo1}, we now turn to the question of classical solvability, as announced in Theorem \ref{theo2}. To this end, we build on the approximate system \eqref{approx} introduced in Section \ref{strong}, together with several of the estimates derived there, and impose accordingly strengthened hypotheses on the initial data and the structural functions $D,C,E,F$. The classical solution asserted in Theorem \ref{theo2} will then arise as a regularity extension of the strong solution already constructed in Section \ref{strong}, similarly to what was done in \cite{FW26}.\abs
Specifically, here and below we assume that, in addition to the hypotheses made in Section \ref{strong}, $D,C\in C^3([0,\infty))$ and that
\bea{5.1}
D_\eps\to D\qquad\text{and}\qquad C_\eps\to C\qquad\text{in }C_{loc}^3([0,\infty))\text{ as }\eps\searrow 0,
\eea
that $u_0,u_{0t}$ and $u_{0tt}$ satisfy the strengthened regularity requirements of \eqref{1.20}, and that the corresponding approximations are chosen so as to comply with
\bea{5.2}
u_{0\eps}\to u_0,\qquad v_{0\eps}\to u_{0t}\qquad\text{in }W^{3,2}(\Om)\qquad\text{and}\qquad w_{0\eps}\to u_{0tt}\qquad\text{in }W^{2,2}(\Om)
\eea
as $\eps\searrow 0$; in particular, $\eps_{\star\star}$ and $T_\star$ retain the meaning assigned to them in Lemma \ref{lem2.6} and Lemma \ref{lem2.7}.
\subsection{Auxiliary estimates}
The only genuinely new estimate we need, beyond those already available from Section \ref{strong}, is an $\eps$-independent bound on the third-order spatial derivative of $\Theta_\eps$. Its proof rests on the diffusive structure of the heat equation in \eqref{approx}, together with the bounds from Corollary \ref{cor2.10}.
\begin{lem}\label{lem3.1}
There exists $K>0$ such that if we let $\eps_{\star\star}$ and $T_\star$ be as in Lemma \ref{lem2.6} and Lemma \ref{lem2.7}, then
\bea{5.10}
\int_0^{T_\star}\io \Theta_{\eps xxx}^2\le K,\qquad\text{for all }\eps\in(0,\eps_{\star\star}).
\eea 
\end{lem}
\proof
Let $h_\eps$ be as in \eqref{2.36.1}. Testing the fourth equation of \eqref{approx} in a straightforward manner, we see that thanks to Young's inequality,
\bea{5.13}
\frac{1}{2}\frac{d}{dt}\io \Theta_{\eps xx}^2&=&-\io \Theta_{\eps xxx}\Theta_{\eps xt}\nn\\
&=&-\io \Theta_{\eps xxx}\cdot\big\{\Theta_{\eps xxx}+2a_\eps a_{\eps x}\big\}\nn\\
&\le&-\frac{1}{2}\io\Theta_{\eps xxx}^2+2\io a_\eps^2 a_{\eps x}^2\qquad\text{for all }t\in(0,\tme)\text{ and }\eps\in(0,1),
\eea
with $a_\eps$ as in Lemma \ref{lem2.9}. By Corollary \ref{cor2.10}, fix $c_1>0$ such that
$$\|v_\eps\|_{W^{2,2}(\Om)}+\|u_\eps\|_{W^{2,2}(\Om)}+\|\Theta_\eps\|_{W^{2,\infty}(\Om)}+\|g_\eps\|_{W^{1,2}(\Om)}\le c_1$$
for all $t\in(0,T_\star)$ and $\eps\in(0,\eps_{\star\star})$. Then, taking $(k_i)_{i\in\{1,\dots,12\}}$ from Lemma \ref{lem2.4} and $k_{18}$ as in Lemma \ref{lem2.9}, we obtain from \eqref{3.27} and \eqref{3.28} that
\bea{nn}
\io h_{\eps x}^2&=&\frac{1}{\tau^2}\io \Big\{D_\eps'(\Theta_\eps)\Theta_{\eps x}v_{\eps x}+D_\eps(\Theta_\eps)v_{\eps xx}+C_\eps'(\Theta_\eps)\Theta_{\eps x}u_{\eps x}+C_\eps(\Theta_\eps)u_{\eps xx}-g_{\eps x}\Big\}^2\nn\\
&\le&\frac{5k_7^2}{\tau^2}\io\Theta_{\eps x}^2v_{\eps x}^2+\frac{5k_2^2}{\tau^2}\io v_{\eps xx}^2+\frac{5k_8^2}{\tau^2}\io \Theta_{\eps x}^2u_{\eps x}^2+\frac{5k_4^2}{\tau^2}\io u_{\eps xx}^2+\frac{5}{\tau^2}\io g_{\eps x}^2\nn\\
&\le& c_2:=\frac{5}{\tau^2}\big\{k_7^2c_1^4+k_2^2c_1^2+k_8^2c_1^4+k_4^2c_1^2+c_1^2\big\}\qquad\text{for all }t\in(0,T_\star)\text{ and }\eps\in(0,\eps_{\star\star})\nn
\eea
and, consequently, from \eqref{3.27}, \eqref{3.28.1} and \eqref{2.60} that
\bea{5.12}
\io a_{\eps x}^2&=&\io \Big\{E_\eps '(\Theta_\eps)\Theta_{\eps x}v_{\eps x}+E_\eps(\Theta_\eps)v_{\eps xx}+F_\eps'(\Theta_\eps)\Theta_{\eps x}h_\eps+F_\eps(\Theta_\eps)h_{\eps x}\Big\}^2\nn\\
&\le& c_3:=4\big\{k_{11}^2c_1^4+k_5^2c_1^2+k_{12}^2k_{18}^2c_1^2|\Om|+k_6^2c_2\big\}\qquad\text{for all }t\in(0,T_\star)\text{ and }\eps\in(0,\eps_{\star\star}).\ \ \
\eea
Combining \eqref{5.13} and \eqref{5.12} with \eqref{2.59.0}, we thus obtain
$$\frac{d}{dt}\io \Theta_{\eps xx}^2+\io \Theta_{\eps xxx}^2\le 4k_{18}^2c_3\qquad\text{for all }t\in(0,T_\star)\text{ and }\eps\in(0,\eps_{\star\star}),$$
which, upon an integration using \eqref{2.35} entails \eqref{5.10}, because $[0,\tme)\ni t\mapsto\io \Theta_{\eps xx}^2$ is continuous according to Lemma \ref{lem2.3}.
\qed
With Lemma \ref{lem3.1} at hand, we can now turn to the remaining higher-order estimates required for the proof of Theorem \ref{theo2}:
\begin{lem}\label{lem3.2}
Let $\eps_{\star\star}$ and $T_\star$ be as in Lemma \ref{lem2.6} and Lemma \ref{lem2.7}. Then there exist $\eps_{\star\star\star}\in(0,\eps_{\star\star})$ and $K>0$ such that
\bea{5.20}
\io w_{\eps xx}^2(\cdot,t)\le K,\qquad\io v_{\eps xxx}^2(\cdot,t)\le K\qquad\text{and}\qquad\io u_{\eps xxx}^2(\cdot,t)\le K,
\eea 
for all $t\in(0,T_\star)$ and $\eps\in(0,\eps_{\star\star\star})$.
\end{lem}
\proof
This follows, under the substitution described in the Introduction, from \cite[Lemma 3.2 and Lemma 3.3]{FW26}, combined with the bound on $\Theta_{\eps xxx}$ provided by Lemma \ref{lem3.1}.
\qed
\subsection{Proof of Theorem \ref{theo2}}
Combining Lemma \ref{lem3.2} with suitable embedding properties available in the one-dimensional setting, we obtain the following regularity upgrade of the strong solution:
\begin{lem}\label{lem3.3}
Let $\eps_{\star\star\star}$ and $T_\star$ be as in Lemma \ref{lem3.2} and Lemma \ref{lem2.7}. Then the strong solution $(u,\Theta,g)$ obtained in Theorem \ref{theo1} has the additional properties that
\bea{5.21}
u\in C^0([0,T_\star];C^2(\ov\Om))\cap L^\infty((0,T_\star);W^{3,2}(\Om)),\label{10.8}\\[1mm]
u_t\in C^0([0,T_\star];C^2(\ov\Om))\cap L^\infty((0,T_\star);W^{3,2}(\Om)),\\[1mm]
u_{tt}\in C^0([0,T_\star];C^1(\ov\Om))\cap L^\infty((0,T_\star);W^{2,2}(\Om)),\\[1mm]
u_{ttt}\in C^0(\ov\Om\times[0,T_\star])\cap L^\infty((0,T_\star);W^{1,2}(\Om)),\label{10.11}
\eea
and $(u,\Theta,g)$ actually solves \eqref{0} in the classical sense in $\Om\times(0,T_\star)$.
\end{lem}
\proof
Arguing exactly as in \cite[Lemma 4.1]{FW26}, now based on Lemma \ref{lem3.2}, yields \eqref{10.8}-\eqref{10.11}.
\qed
Theorem \ref{theo2} now follows by the same extension argument and generalization to arbitrary initial data as in the proof of Theorem \ref{theo1}.\abs\abs
{\sc Proof of Theorem \ref{theo2}.\quad}For initial data satisfying \eqref{---}, we let
$$\tm:=\sup\Big\{T>0\ \Big|\ \text{There exists a classical solution of \eqref{0} in }\Om\times(0,T)\Big\}.$$ By Lemma \ref{lem2.2}, there is exactly one classical solution $(u,\Theta,g)$ of \eqref{0} in $\Om\times(0,\tm)$. Suppose this solution violated \eqref{1.25}. Then $\tm<\infty$ and
$$\|u_t\|_{W^{3,2}(\Om)}+\|u_{tt}\|_{W^{2,2}(\Om)}+\|\Theta\|_{W^{2,\infty}(\Om)}\le c_1\qquad\text{for all }t\in(t_0,\tm)$$ with some $c_1>0$ and some $t_0\in(0,\tm)$. Thus, analogously as in the proof of Theorem \ref{theo1}, $u$ remains bounded in $W^{3,2}(\Om)$ and $g$ in $W^{1,2}(\Om)$ on $(t_0,\tm)$. On this basis we may fix $t_j\nearrow\tm$ as $j\to\infty$, along which
\be{nn}
	\lball
	u(\cdot,t_j)\rightharpoonup \tilde{u}_0&\qquad\text{in }W^{3,2}(\Om)\qquad\text{as }j\to\infty,\\[1mm]
	u_t(\cdot,t_j)\rightharpoonup\tilde{u}_{0t}&\qquad\text{in }W^{3,2}(\Om)\qquad\text{as }j\to\infty,\\[1mm]
	u_{tt}(\cdot,t_j)\rightharpoonup\tilde{u}_{0tt}&\qquad\text{in }W^{2,2}(\Om)\qquad\text{as }j\to\infty,\\[1mm]
	g(\cdot,t_j)\rightharpoonup\tilde{g}_0&\qquad\text{in }W^{1,2}(\Om)\qquad\text{as }j\to\infty,\qquad\text{and that}\\[1mm]
	\Theta(\cdot,t_j)\stackrel{\ast}{\rightharpoonup}\tilde{\Theta}_0&\qquad\text{in }W^{2,\infty}(\Om)\qquad\text{as }j\to\infty,
	\ear
\ee
with $\tilde{u}_0, \tilde{u}_{0t}\in W^{3,2}(\Om)$, $\tilde{u}_{0tt}\in W^{2,2}(\Om)$, $\tilde{g}_0\in W^{1,2}(\Om)$ and $\tilde{\Theta}_0\in W_N^{2,\infty}(\Om;[0,\infty))$. Applying Lemma \ref{lem2.13} and Lemma \ref{lem3.3} to these data, we obtain $T>\tm$ together with a classical solution of \eqref{0} in $\Om\times(\tm,T)$ attaining these data at $t=\tm$; gluing this classical solution to $(u,\Theta,g)$ across $t=\tm$ then produces a classical solution of \eqref{0} on all of $\Om\times(0,T)$, contradicting the maximality of $\tm$.\abs
For general initial data, the reduction proceeds exactly as in the proof of Theorem \ref{theo1}: writing $\tilde{u}_0:=u_0-y_0$, $\tilde{u}_{0t}:=u_{0t}-y_1$ and $\tilde{u}_{0tt}:=u_{0tt}-y_2$, with $y_0$, $y_1$ and $y_2$ as in \eqref{4.1}, applying the above to these normalized data, and adding back the polynomial correction $y(t)$ solving $\tau y'''(t)+y''(t)=0$ in $(0,\infty)$ with $y(0)=y_0$, $y'(0)=y_1$ and $y''(0)=y_2$, yields the desired classical solution of \eqref{0} on $\Om\times(0,\tm)$ satisfying \eqref{1.21}-\eqref{1.25}.
\qed
\bigskip
{\bf Acknowledgment.} \quad
The author acknowledges support of the Deutsche Forschungsgemeinschaft (Project No. 444955436).
He moreover declares that he has no conflict of interest.\abs
{\bf Data availability statement.} \quad
Data sharing is not applicable to this article as no datasets were
generated or analyzed during the current study.

\end{document}